\documentclass[11pt,reqno]{amsart}
\usepackage{amssymb,amsmath}
\usepackage{amsmath,amstext,amssymb,amscd}
\usepackage{verbatim}
\usepackage{enumerate}
\usepackage{mathrsfs}
\usepackage[dvipsnames,usenames]{xcolor}

\usepackage{amsthm}
\usepackage[colorlinks=true,linkcolor=blue]{hyperref}

\def\Bbb{\mathbb}
\def\Cal{\mathcal}
\def\Dt{\partial_t}

\def\eb{\varepsilon}

\def\R {\mathbb{R}}

\def\<{\left<}
\def\>{\right>}

\def\Nx{\nabla_x}
\def\Dx{\Delta_x}
\def\({\left(}
\def\){\right)}
\def\divv{\operatorname{div}}

\newtheorem{proposition}{Proposition}[section]
\newtheorem{theorem}[proposition]{Theorem}
\newtheorem{corollary}[proposition]{Corollary}
\newtheorem{lemma}[proposition]{Lemma}

\theoremstyle{definition}
\newtheorem{definition}[proposition]{Definition}

\newtheorem{remark}[proposition]{Remark}
\newtheorem{example}[proposition]{Example}

\numberwithin{equation}{section}

\def \no#1#2#3 {{\bf #1} (#3), #2.}
\def \eds#1#2#3 {#1, #2, #3.}

\title[Inertial manifolds and spatial averaging]
{Inertial manifolds via  spatial averaging revisited}
\author[A. Kostianko, X. Li, C. Sun and S. Zelik] {Anna Kostianko${}^{\dag,\ddag}$, Xinhua Li${}^\dag$, Chunyou Sun${}^\dag$, and Sergey Zelik${}^{\dag,\ddag}$}
\begin{document}
\begin{abstract} The paper gives a comprehensive study of inertial manifolds for semilinear
parabolic equations and their smoothness using the spatial averaging method suggested by G. Sell and
 J. Mallet-Paret. We present a universal approach which covers the most part of known results obtained
  via this method as well as gives a number of new ones. Among our applications are
   reaction-diffusion equations, various types of generalized  Cahn-Hilliard equations, including fractional
    and 6th order Cahn-Hilliard equations and several classes of modified Navier-Stokes equations including
     the Leray-$\alpha$ regularization, hyperviscous regularization and their combinations. All of the results are obtained in 3D case with
      periodic boundary conditions.
\end{abstract}

\address{${}^\dag$ \phantom{e}School of Mathematics and Statistics, Lanzhou University, Lanzhou  \\ 730000,
P.R. China}
\email{xhli2014@lzu.edu.cn (X. Li), sunchy@lzu.edu.cn (C. Sun)}
\address{${}^\ddag$ University of Surrey, Department of Mathematics, Guildford, GU2 7XH, United Kingdom.}
\email{aNNa.kostianko@surrey.ac.uk (A. Kostianko), s.zelik@surrey.ac.uk (S. Zelik)}


\thanks{ This work is partially supported by the grant 19-71-30004 of RSF,
the EPSRC grant EP/P024920/1 and NSFC grant No. 11522109, 11871169.}
\date{May 25, 2020}
\maketitle

\tableofcontents
\section{Introduction}\label{sec1}

It is believed that in many cases the longtime behavior of trajectories of a dissipative system,
say, generated by a partial differential equation (PDE) is essentially finite-dimensional. In other words,
despite of the  infinite-dimensionality of the initial  phase space, the generated dynamics is governed,
up to some "non-essential" transient effects, by finitely many parameters, the so-called {\it order}
parameters in the terminology of I. Prigogine, see \cite{Pri77}. Ideally, it is expected that these
order parameters obey a system
of ordinary differential equations (ODEs) which is called an inertial form (IF) of the initial
dissipative system. Thus, the IF if it exists allows us to reduce the study of the
dynamics generated by PDEs to the study of the corresponding system of ODEs  which in turn can be done
 using the methods of classical dynamics. In particular, the dream to understand the nature of turbulence
  using the ideas and methods of classical dynamics permanently inspires the development of the dynamical
  theory of dissipative systems during the last 50 years, see \cite{BV92,CV02,Fef06,FP67,F95,Lions1969,R01,T95,T97} and references therein.
We only mention here that the key concepts of the theory like {\it inertial} form or
 {\it inertial} manifold were initially related with the so-called  {\it inertial} scale in the theory of
  turbulence and the corresponding {\it inertial} term in Navier-Stokes equations.
\par
However, despite a lot of progress done by prominent researches, the nature of the above mentioned
finite-dimensional reduction and its rigorous justification somehow remains a mystery. Moreover, as recent examples
 and counterexamples show, there are deep obstacles to effective realization of this program, e.g., related
  with the {\it smoothness} of the IF and related finite-dimensional reduction, see \cite{Z14}
  and references therein.
\par
Indeed, the most popular way to justify this finite-dimensional reduction is related with the
 theory of attractors. By definition, a global attractor of a dynamical system (DS) is a compact
  invariant set in the phase space which attracts the images of bounded sets as time tends to infinity.
The main achievement of the attractors theory is that a global attractor $\Cal A$ exists under rather weak
 assumptions on a dissipative system considered and in many cases has
finite Hausdorff and box-counting dimensions, see  \cite{BV92,CV02,R01,T97,MZ08} and references
 therein. The class of
 such systems includes reaction-diffusion and 2D Navier-Stokes systems, pattern formation equations
 (like Cahn-Hilliard or Swift-Hohenberg ones),
  damped wave equations and many others. This result in turn allows us to build up the desired
  finite-dimensional
   reduction as well as the IF using the Man\'e projection theorem, see \cite{R11}
    and references therein. In this approach the box-counting dimension of the attractor $\Cal A$ is usually
     interpreted as a number of "degrees of freedom" in the reduced IF. In particular, this explains the
      permanent interest to various upper and lower bounds for the box-counting dimension of $\Cal A$.
      \par
      On the other hand, the obtained in such a way IF is only H\"older continuous and it can be not
       even Lipschitz continuous in general. In a fact, there are natural examples where the box-counting
        dimension
        of the attractor is low (e.g, 3), but a Lipschitz IF does not exist. Moreover, the reduced
         dynamics on the
         attractor contains features which can hardly be interpreted as "finite-dimensional" (like limit
         cycles with super-exponential rate of attraction, traveling waves in Fourier space, etc.),
         see \cite{EKZ13,MPSS93,Rom00,KZ18,Z14} for more details. In these cases, the "finite-dimensionality" obtained
          via Man\'e projections looks artificial and controversial and it seems more natural to
           accept that the dynamics here is infinite-dimensional despite the finiteness
            of box-counting dimension.
\par
The above mentioned problems motivate an increasing interest to alternative
methods of constructing IFs, not related with box-counting dimension and Man\'e projection theorem. One of
 the most natural alternative approaches is based on the concept of an {\it inertial}
  manifold (IM) suggested in \cite{FST88}. Roughly speaking, an IM $\Cal M$ is a sufficiently smooth
   (at least Lipschitz)
   finite-dimensional invariant submanifold of the phase space which is normally-hyperbolic
   and exponentially stable. If such an object exists, then the finite-dimensional reduction is
    ideally justified. Indeed, the reduction of the initial PDE to the manifold $\Cal M$ gives us the
     desired IF and the normal hyperbolicity gives us the so-called asymptotic phase or exponential tracking
     property which in turn gives us a nice rigorous interpretation in what sense the
      transient features are "non-essential".
\par
However, being a sort of a center manifold, an IM requires strong separation of the phase space on slow
 and fast variables which is usually stated in the form of {\it spectral gap} conditions or/and
 invariant cone properties, see \cite{CL02,CFNT89,FST88,FNST88,Fen72,MS89} and references therein
 for more details. In particular,
 for the simplest model of a semilinear parabolic equation in a real Hilbert space $H$:
 \begin{equation}\label{0.rde}
\Dt u+Au+F(u)=0,
 \end{equation}
where $A: D(A)\to H$ is a positive self-adjoint operator with compact inverse and $F:H\to H$ is a globally
Lipschitz map with Lipschitz constant $L$, the spectral gap conditions for existence of $N$-dimensional
 IM read:
\begin{equation}\label{0.sg0}
\lambda_{N+1}-\lambda_N>2L,
\end{equation}
where $\{\lambda_n\}_{n=1}^\infty$ are the eigenvalues of the operator $A$
enumerated in a non-decreasing order. In the present paper we are mainly interested in a more
complicated version of the abstract parabolic problem, namely,
\begin{equation}\label{0.rdeg}
\Dt A^{-\gamma}u+Au+F(u)=0,
\end{equation}
where $\gamma\ge0$ and $A$ and $F$ are the same as in \eqref{0.rde}. The spectral gap conditions
 for this equation read:
\begin{equation}
\frac{\lambda_{N+1}^{1+\gamma}-\lambda_N^{1+\gamma}}{\lambda_{N+1}^\gamma+\lambda_N^\gamma}>L.
\end{equation}
It is known that these spectral gap conditions are sharp in the sense that if they are not satisfied
 one always can construct a nonlinearity $F$ for which the corresponding IM will not exist,
  see \cite{EKZ13,M91,R94,Rom00,Z14} for more details. Thus, there is no hope to push forward the
   theory beyond
  the spectral gap conditions at least on the level of general abstract nonlinearities. However, this is
   possible for some partial classes of operators $A$ and nonlinearities $F$ (see \cite{M-PS88} and \cite{KZ15}
    for 3D reaction-diffusion and Cahn-Hilliard equations with periodic boundary conditions,
    \cite{KZ18,KZ17} for 1D
     reaction-diffusion-advection problems, \cite{K18,LS20,GG18} for modified Navier-Stokes equation
      and \cite{K20}
     for the complex Ginzburg-Landau equation).
\par
In the present paper, we are mainly interested in the so-called spatial averaging method which
 has been introduced in \cite{M-PS88} in order to verify the existence of an IM for 3D scalar
  reaction-diffusion equation with periodic boundary conditions, see also \cite{Z14} for more
   recent exposition of the theory and \cite{Kwe99} for slightly different boundary conditions. Roughly speaking, the method works in the case where the
   derivative $F'(u)$ contains {\it point-wise multiplication} and utilizes some special features of such
   multiplication operators which comes from harmonic analysis and number theory, see \cite{M-PS88,Z14} and
    Section \ref{s5} below for more details. These features allow us to replace in the analysis the
     multiplication on a function by the {\it scalar} operator of multiplication
      on its {\it spatial average} (which explains the name of the method). This trick essentially
       simplifies the analysis and allows us to go beyond of spectral gap conditions at least
        in the case of 3D problems with periodic boundary conditions.  Note also that in general
        this method does not work for systems since we will have not a scalar operator,
         but matrix operator instead and this is not enough for IMs, so some further steps
          are necessary, see \cite{K20} for the case of complex Ginzburg-Landau equation where the
           spatial averaging is combined with the temporal one in order
           to get finally the scalar operator. But there is an important exception pointed
            out in \cite{K18}, namely, the case of zero spatial averaging which is typical
            for the Navier-Stokes type nonlinearities and which allowed to treat the modified
             Navier-Stokes equations using the spatial averaging method, see also \cite{GG18,LS20}.
\par
The aim of the present paper is to give a systematic study of IMs via the spatial averaging based
 on the universal model \eqref{0.rdeg} which allows to treat most part of known applications of spatial
  averaging technique as well as to get new ones from the unified point of view. Among the
   considered applications are classical reaction-diffusion equations,
   various types of Cahn-Hilliard (CH) equations, including the so-called
    fractional CH, 6th order CH, etc., and various modifications of Navier-Stokes equations
     including the Bardina model and Leray $\alpha$-model, hyperviscous Navier-Stokes and their
     combinations.
The paper is organized as follows.
\par
In Section \ref{s1} we  discuss the analytic properties (such as existence
and uniqueness of solutions, their regularity
 and various versions of a parabolic smoothing property) of solutions of problem
  \eqref{0.rdeg} with globally
 Lipschitz nonlinearity $F$. These properties will be used throughout of the paper.
 \par
 In Section \ref{s2}, we recall (following mainly  \cite{KZ15} and \cite{Z14}) the strong cone property
  in a differential form and general theorems about existence of an IM
   of regularity $C^{1+\eb}$ with $\eb>0$ adapted to the case of equation \ref{0.rdeg}.
   \par
 Verification of the strong cone property based on an abstract version of spatial averaging
  introduced in \cite{KZ15} is given in Section \ref{s3}. In particular, we present here the abstract
   theorems on the existence and $C^{1+\eb}$-smoothness of an IM for equation \eqref{0.rdeg}
    (also in the spirit of \cite{KZ15}).
\par
We note that usually most part of equations interesting from the applied point of view do not
 have nonlinearities which are globally Lipschitz in $H$, so, in order to get an IM for such equations,
 one usually first verify the existence of a good absorbing/attracting set in the phase space and
  then {\it truncate} the nonlinearity outside of this absorbing set to end up with globally
   Lipschitz nonlinearity. This truncation procedure is usually simple in the case when the
   spectral gap conditions are satisfied, but may be very delicate in the case of spatial averaging
   since the truncation should not affect much the spatial averaging property for the nonlinearity.
   For instance, in the original paper \cite{M-PS88} where the spatial averaging method has been suggested,
   the authors have to truncate not only the nonlinearity, but also to change in a very non-trivial way
   the leading part $Au$ of the equation. Analogously, the applications of spatial averaging
    to Navier-Stokes equations become possible due to the special truncation function $W(u)$ which
     truncates the Fourier modes of the solution $u$, suggested in \cite{K18}, see Section \ref{s4}
      for more details.
\par
In Section \ref{s4}, we suggest a unified truncation procedure (which somehow combines the approaches
 developed in \cite{M-PS88} and \cite{K18}) which allows us to deduce the spatial averaging property
  for the truncated nonlinearity directly from some natural properties
   (Assumptions I-III, see Section \ref{s4})
   of the initial non-truncated nonlinearity and the extra assumption that the initial
    non-truncated equation possesses an absorbing set in a "good" space.
\par
In Section \ref{s5} we restrict ourselves to the case where $A$ is the Laplacian in a 3D domain
 $(-\pi,\pi)^3$ with periodic boundary conditions and verify the spatial averaging property for all classes
  of nonlinearities important for our applications as well as other of Assumptions I-III. Thus,
  to get the existence
   of $C^{1+\eb}$-smooth IM, it only remains to verify the global well-posedness of the problem and the
   existence of an absorbing set in the proper "good" space.
   \par
This verification is finally done in Section \ref{s6}. Namely, the application of our
method to the classical scalar reaction-diffusion equation:
\begin{equation}\label{0.exrde}
\Dt u=\Dx u-u+f(u)+g,\ \ u\big|_{t=0}=u_0
\end{equation}
in a 3D domain endowed with periodic boundary conditions is given in subsection \ref{ss7.1}.
 We assume that $g\in H=L^2(\Omega)$ and $f$
satisfies the assumptions
 \begin{equation}\label{0.f}
1.\ f\in C^4(\R,\R),\ \
2.\ f(u)u\ge -C,\ \
3.\ f'(u)\ge-K,\ \ u\in\R.
 \end{equation}
This equation formally fits to equation \eqref{0.rdeg} with $\gamma=0$ and the main result is that
 under assumptions \eqref{0.f} this equation possesses a $C^{1+\eb}$-smooth IM, for instance,
  in the phase space $H$. In this case our approach gives nothing new in comparison with
   the standard results (and it is even a bit weaker since more accurate analysis shows
    that $f$ may be taken to be $C^2$-smooth only), but it is nevertheless presented here in order to
     demonstrate that this classical result is covered by our unified scheme.
\par
Subsection \ref{ss7.2} is devoted to the generalizations of the Cahn-Hilliard equations, namely,
\begin{equation}\label{0.exCH}
\Dt u+(-\Dx)^{\gamma}(-\Dx u+f(u)+g)=0,\ \ u\big|_{t=0}=u_0,\ \ \gamma>0
\end{equation}
in a 3D domain $(-\pi,\pi)^3$ with periodic boundary conditions. Due to the presence of the mass
 conservation law it is natural to consider this equation in the spaces of functions with zero mean
 $$
 \<u\>:=\frac1{(2\pi)^3}\int_{(-\pi,\pi)^3}u(x)\,dx,
 $$
 for instance $H=\{u\in L^2((-\pi,\pi)^3),\ \<u\>=0\}$. Then the Laplacian is positive definite and this
  equation indeed has the form of \eqref{0.rdeg}. The choice $\gamma=1$ corresponds to the
  classical Cahn-Hilliard equation considered in \cite{KZ15}. The choice $\gamma\in(0,1)$ gives the
   so-called
  fractional Cahn-Hilliard equation (see \cite{ASS16}) and $\gamma=2$ gives us the so-called 6th order
   Cahn-Hilliard equation, see \cite{Mir14} and references therein. To the best of our knowledge the
    questions related with IMs for the last two equations have been not considered in the literature.
    \par
  As an application of our abstract scheme, we get the existence of $C^{1+\eb}$-smooth IM for
  equation \eqref{0.exCH}
  for all $\gamma>0$, $g\in H$ and $f$ satisfying \eqref{0.f}. We also note that the natural phase
   space for problem \eqref{0.exCH} as well as for our abstract model \eqref{0.rdeg}
    is $H^{-\gamma}$. However, due to the smoothing property for differences of solutions
    verified in Section \ref{s1}, the statements about the existence of IM are {\it equivalent}
    in all spaces between which this smoothing property holds. By many reasons,
    it is more convenient to verify the existence of an IM in the space $H^{-\gamma}$ and
    then to extend it to
    all phase spaces $H^s$, $-\gamma<s<2$ using the above mentioned smoothing property.
\par
Finally, the case of Navier-Stokes type nonlinearities is considered in subsection \ref{ss7.3}.
Note that the classical 3D Navier-Stokes is out of reach of the modern theory even from the point
 of view of global well-posedness of solutions, so using some modified models looks unavoidable
  at this stage.
 In addition, existence of an IM even for the 2D case
  (where the global well-posedness is well-known)
is one of the key open problems in the field, so in order to get the existence of IMs we need stronger
 modifications. In this paper, we consider the following combination of hyper-viscosity with
 Leray-$\alpha$ type regularization of the velocity vector field:
 \begin{equation}\label{0.exNS}
\begin{cases}
\Dt u+(u,\Nx)\bar u+\Nx p+(-\Dx)^{1+\gamma} u=g,\ \ u\big|_{t=0}=u_0,\\
\divv u=0,\ \ \bar u:=(1-\alpha\Dx)^{-\bar\gamma}u,
\end{cases}
 \end{equation}
where  $u=(u_1,u_2,u_3)$, $\bar u=(\bar u_1,\bar u_2,\bar u_3)$ and $p(t,x)$ are unknown velocity,
"filtered" velocity and pressure respectively, $g$ is a given external forces, $\alpha>0$ is a given
 length scale parameter, $\gamma\ge0$ is a given hyper-viscosity exponent
and  a given parameter  $\bar\gamma\ge 0$ affects the
 strength of the nonlinear term.
\par
Various regularisations of the initial Navier-Stokes equations (including \eqref{0.exNS})
 have been intensively studied after the
 pioneering work of J. Leray \cite{L34} by many researches, see
 \cite{A13,BFR80,CHOT05,HLT10,LL06,ILT06,OT07} and references therein. In particular, in order to
 guarantee the global well-posedness of problem \eqref{0.exNS}, we need to require that
 \begin{equation}
2\gamma+\bar\gamma\ge\frac12,
 \end{equation}
 see also subsection \ref{ss7.3} for more details.
\par
The existence of an IM for this problem in 2D case with periodic boundary conditions
 for $\gamma=0$ and $\bar\gamma=1$ has been
 verified in \cite{HGT15} using the spectral gap conditions (which hold in 2D case but fail in 3D).
 The spatial averaging method has been applied instead of spectral gap conditions in \cite{K18} in order
  to treat 3D case with the same parameters $\gamma=0$ and $\bar\gamma=1$. The possibility to treat the
   "double-critical" case $\gamma=0$, $\bar\gamma=\frac12$ has been also outlined in \cite{K18} and then
    verified in details in \cite{LS20}. The purely hyperviscous case $\gamma=\frac12$ and
     $\bar\gamma=0$ has studied in \cite{GG18}.
\par
In the present paper, we give the existence of the $C^{1+\eb}$-smooth IMs for all intermediate cases.
Namely, as we will see below, the spatial averaging technique works if $\gamma+\bar\gamma\ge\frac12$. The
 case of strict inequality is usually simpler and can be treated using the spectral gap conditions
  (if $\gamma>0$),
 so we concentrate on the critical (from the point of view of IMs) case when
 \begin{equation}\label{0.IMc}
\gamma+\bar\gamma=\frac12,\ \ \gamma\in[0,\frac12].
\end{equation}
In this case, we define the basic space
$$
H:=\{u\in [L^2((-\pi,\pi)^3)]^3, \ \divv u=0,\ \ \<u\>=0\}
$$
and the operator $A$ as a Stokes operator (=Laplace operator restricted to the invariant
subspace of divergent free vector fields with zero mean). Then, applying the operator $A^{-\gamma}$
 to both sides of \eqref{0.exNS}, we  get the equation of the form \eqref{0.rdeg} and may apply our
 general theory to the obtained equation. This gives us the following result: for every $\gamma$ and
 $\bar\gamma$ satisfying \eqref{0.IMc} and every external forces $g$ such that $A^{-\gamma}g\in H$,
 there exist a $C^{1+\eb}$-smooth IM for the problem \eqref{0.exNS}. For the end points, this result
  covers the
  results obtained before, but it seems new for all intermediate cases. In addition, all the
   previous results for
   this equation give only Lipschitz continuous IM and $C^{1+\eb}$-smoothness is also a novelty.
\par
We finally note that, analogously to the case of reaction-diffusion equations, our result is applicable and
gives new results in the 2D case as well. Indeed, in the case of {\it square} torus $(-\pi,\pi)^2$, we have
 the spectral
 gaps of length $\ln\lambda_N$ in the spectrum of the Laplace or Stokes operator, see \cite{R82},
  so the spectral
 gap condition will
  be satisfied and no spatial averaging is required. However, this result is not known in the case of
   rectangular torus $(-\pi,\pi)\times(-\beta\pi,\beta\pi)$ if $\beta$ is irrational. In this case,
   the spatial averaging works and allows us to overcome the problem and get the desired IM.

\section{Preliminaries and an abstract model}\label{s1}
In this section we recall some basic notations, introduce an
 abstract model equation which will be of our main interest throughout  the paper and prove
 some elementary, but useful properties of its solutions. Let $A: D(A)\to H$ be a positive definite
 self-adjoint
  operator in a Hilbert space $H$ with compact inverse and let $\lambda_1\le\lambda_2\le\cdots$ be its
  eigenvalues
  enumerated in the non-decreasing order. The corresponding orthonormal base   in $H$ generated by
  its eigenvectors
   will be denoted by $\{e_n\}_{n=1}^\infty$. Then any element $u\in H$ is presented by its Fourier series:
\begin{equation}\label{01.F}
u=\sum_{n=1}^\infty u_n e_n,\ u_n:=(u,e_n),\ \ \|u\|^2_H=\sum_{n=1}^\infty u_n^2,
\end{equation}
where $(u,v)$ is the inner product in the space $H$.
\par
The fractional powers $A^s$, $s\in\R$ of operator $A$ are defined using  the standard formula
\begin{equation}
A^su:=\sum_{n=1}^\infty \lambda_n^s (u,e_n)e_n
\end{equation}
and the spaces $H^s:=D(A^{s/2})$ are defined as completions of finite linear combinations
 of $\{e_n\}_{n=1}^\infty$ with respect to the norm
\begin{equation}
\|u\|_{H^s}^2=\|A^{s/2}u\|^2_{H}=(A^su,u)=\sum_{n=1}^\infty\lambda_n^su_n^2.
\end{equation}
We consider the following abstract semi-linear parabolic problem in $H$:
\begin{equation}\label{1.FCH}
\Dt A^{-\gamma}u+Au+F(u)=g,\ \ u\big|_{t=0}=u_0,
\end{equation}
where $\gamma\ge0$ is a fixed exponent and $F:H\to H$ is a given nonlinearity which is assumed to be
globally bounded
\begin{equation}\label{1.FBOUND}
\|F(u)\|_H\le C
\end{equation}
and is globally Lipschitz continuous with global Lipschitz constant $L$:
\begin{equation}\label{1.FLIP}
\|F(u_1)-F(u_2)\|_{H}\le L\|u_1-u_2\|_{H},\ \ u_1,u_2\in H.
\end{equation}
The external force $g$ is time-independent and is taken from the space $H$ ($g\in H$).
\par
The natural phase space for problem \eqref{1.FCH} is $\Phi:=H^{-\gamma}$ ($u_0\in H^{-\gamma}$)
 although as we see from the next proposition the solution $u(t)$ becomes at least $H^2$-smooth at any
  positive time $t>0$. As usual the solutions are understood in the sense of distributions, namely,
   $u\in C(0,T;H^{-\gamma})\cap L^2(0,T;H^1)$ is a solution of \eqref{1.FCH} if for every test function
    $\varphi\in C_0^\infty(0,T;H^{2})$, the following identity holds:
    \begin{equation}
-\int_\R(u(t),A^{-\gamma}\Dt \varphi(t))\,dt+\int_\R(u(t),A\varphi(t))\,dt=\int_\R(g-F(u),\varphi(t))\,dt.
    \end{equation}

\begin{proposition}\label{Prop01.well} Let the nonlinearity $F$ satisfy \eqref{1.FBOUND} and
 \eqref{1.FLIP} and the external force $g\in H$. Then
\par
1. Equation \eqref{1.FCH} is uniquely globally solvable for all $u_0\in H^{-\gamma}$ and the corresponding
solution operators $S(t):H^{-\gamma}\to H^{-\gamma}$, $t\ge0$, generate a dissipative
 semigroup in $H^{-\gamma}$, i.e., the following estimate holds:
\begin{equation}\label{01.dis}
\|u(t)\|^2_{H^{-\gamma}}+\|u\|^2_{L^2(t,t+1;H^1)}\le
 Ce^{-\alpha t}\|u_0\|^2_{H^{-\gamma}}+C(1+\|g\|^2_{H}),
\end{equation}
where $u(t):=S(t)u_0$ and positive constants $C$ and $\alpha$ are independent of $t$ and $u_0$.
\par
2. The constructed semigroup $S(t)$ is globally Lipschitz continuous in $H^{-\gamma}$, i.e.,
for every two solutions $u_1(t)$ and $u_2(t)$ of equation \eqref{1.FCH}, we have
\begin{equation}\label{01.lip}
\|u_1(t)-u_2(t)\|^2_{H^{-\gamma}}+\|u_1-u_2\|^2_{L^2(t,t+1;H^1)}\le
C\|u_1(0)-u_2(0)\|^2_{H^{-\gamma}}e^{L_\gamma t},
\end{equation}
where the positive constants $C$ and $L_\gamma$ depend only on $L$ and $\gamma$.
\par
3. The semigroup $S(t)$ possesses an instantaneous $H^{-\gamma}$ to $H^2$ parabolic smoothing property, i.e.,
\begin{equation}\label{01.h2-sm}
\|u(t)\|_{H^2}\le C t^{-1}(\|u(0)\|_{H^{-\gamma}}+\|g\|_{H}+1),\ \ t\in(0,1]
\end{equation}
where the positive constant $C$  depends on $\gamma$, $A$ and $F$ only. In addition, if we know that
 $u(0)=u_0\in H^2$, then we have the dissipative estimate in $H^2$ as well:
 \begin{equation}\label{01.h2-dis}
\|u(t)\|^2_{H^2}\le C\|u(0)\|^2_{H^2}e^{-\alpha t}+C(1+\|g\|^2_H).
 \end{equation}
\end{proposition}
\begin{proof} Since all statements of this proposition are more or less standard and can be checked as
 in the linear case $F=0$, we give here only the sketch of the proof and leave the details for the reader.
 \par
 {\it Step 1. A priori estimate in $H^{-\gamma}$.} To this end, we multiply (take an inner product)
 of equation \eqref{1.FCH} with $u$ (it is easy to see that all obtained terms make sense, so this multiplication
  is justified). This gives
  $$
  \frac12\frac d{dt}\|u(t)\|^2_{H^{-\gamma}}+\|u(t)\|^2_{H^1}+(F(u),u)=(g,u).
  $$
  Using the  inequality $\|u\|^2_{H^{-\gamma}}\le\lambda_1^{-\gamma-1}\|u\|_{H^1}^2$, the boundedness of $F$ and
   the Gronwall lemma, we get the desired dissipative estimate \eqref{01.dis}
\par
{\it Step 2. Existence and uniqueness.} Let $u_1$ and $u_2$ be two solutions and $v(t)=u_1(t)-u_2(t)$. Then this
 function solves
\begin{equation}
\Dt A^{-\gamma}v(t)+Av(t)+[F(u_1(t))-F(u_2(t))]=0.
\end{equation}
Multiplying this equation by $v$ and using the Lipschitz continuity of $F$, we get
$$
\frac12\frac d{dt}\|v(t)\|^2_{H^{-\gamma}}+\|v(t)\|^2_{H^1}\le L\|v(t)\|^2_{H}.
$$
Using the obvious interpolation inequality
$$
\|v\|^2_H\le \eb\|v\|^2_{H^{1}}+C_\eb\|v\|^2_{H^{-\gamma}}
$$
and the Gronwall lemma, we get the desired uniqueness and estimate \eqref{01.lip}.
\par
The existence of a solution can be obtained by the standard Galerkin approximations using, e.g., the
spectral base $\{e_n\}_{n=1}^\infty$, see e.g. \cite{BV92,T97} for the details.
\par
{\it Step 3. Estimates for $\Dt u$.} We first note that expressing $\Dt u$ from equation \eqref{1.FCH} and
 using estimates \eqref{01.dis}, we conclude that
 \begin{equation}\label{1.Dtu-new}
\|\Dt u\|_{L^2(0,1;H^{-2\gamma-1})}\le C(1+\|g\|_{H}+\|u_0\|_{H^{-\gamma}}).
 \end{equation}
After that, formally differentiating equation \eqref{1.FCH} in time and denoting $v(t):=\Dt u(t)$,
we get the equation
\begin{equation}
\Dt A^{-\gamma}v+Av+F'(u(t))v=0.
\end{equation}
Multiplying this equation by $t^2v(t)$, we arrive at
\begin{equation}\label{01.in-sm}
\frac12\frac d{dt}(t^2\|v(t)\|^2_{H^{-\gamma}})+t^2\|v(t)\|^2_{H^1}\le
 Lt^2\|v(t)\|^2_H+t\|v(t)\|^2_{H^{-\gamma}}.
\end{equation}
We estimate the last term in the right-hand side using the interpolation inequality:
$$
t\|v(t)\|^2_{H^{-\gamma}}\le Ct\|v(t)\|_{H^{-2\gamma-1}}\|v(t)\|_{H^1}\le
 \frac14t^2\|v(t)\|^2_{H^1}+C\|v(t)\|^2_{H^{-2\gamma-1}}.
$$
Integrating this inequality over $t$ and using the obvious inequality
$$
Lt^2\|v(t)\|^2_H\le \frac14t^2\|v(t)\|^2_{H^1}+C\|v(t)\|_{H^{-2\gamma-1}}^2,
$$
we end up (using also \eqref{1.Dtu-new}) with the desired inequality
\begin{equation}\label{01.sm-dt}
t^2\|\Dt u(t)\|^2_{H^{-\gamma}}\le C(1+\|u_0\|^2_{H^{-\gamma}}+\|g\|^2_H),\ \ t\in[0,1].
\end{equation}
Being pedantic, estimate \eqref{01.sm-dt} requires justification. This justification can be
done by approximating the solution $u$ by spectral Galerkin solutions $u_N(t)$ and on the
 finite-dimensional level the corresponding nonlinearity which is a priori Lipschitz can be easily
  approximated by smooth ones without increasing the Lipschitz constant. Since all these arguments are standard,
  we left the details to the reader.
\par
{\it Step 4. Smoothing property for $u(t)$.} We rewrite equation \eqref{1.FCH} as a point-wise in $t$ elliptic problem:
\begin{equation}\label{01.el-l}
Au(t)=\tilde g(t):=g-A^{-\gamma}\Dt u(t)-F(u(t))
\end{equation}
which together with the already obtained estimate for $\Dt u(t)$ gives the desired estimate
 \eqref{01.h2-sm} for $u(t)$. As an immediate corollary of \eqref{01.h2-sm} and \eqref{01.dis} we get
  the desired dissipative estimate \eqref{01.h2-dis} for large enough $t$ (say, $t\ge1$).
\par
{\it Step 5. $H^2$-estimates for small time.} As usual for Cahn-Hilliard type equations, there is an
extra small problem to get estimates of $\|u(t)\|_{H^2}$ for finite (small) time $t>0$. The above
 technique based on estimating $\Dt u(t)$ does not work well here since $u_0\in H^2$ is not enough to get $\Dt u(0)\in H^{-\gamma}$, so we need
  to argue in a bit more delicate way. Namely, we will use the classical parabolic regularity
  stated in the following lemma.
  \begin{lemma} Let $u(t)$ solve the linear problem:
  \begin{equation}\label{01.lin}
\Dt u+A^{1+\gamma}u=h(t),\ u\big|_{t=0}=u_0\in H^2,\ \ h\in C^\kappa(0,1;H^{-2\gamma})
  \end{equation}
  for some $0<\kappa\le\frac1{2(\gamma+1)}$.
  Then the following  estimate holds:
\begin{equation}\label{01.mreg-g}
\|u\|_{C^1(0,1;H^{-2\gamma})\cap C(0,1;H^2)}+\| u\|_{C^\kappa(0,1; H^1)}\le
 C(\|u_0\|_{H^2}+\|h\|_{C^\kappa(0,1;H^{-2\gamma})}).
\end{equation}
\end{lemma}
\begin{proof} We split $u(t)=u_1(t)+u_2(t)$, where
$$
\Dt u_1+A^{1+\gamma}u_1=h,\ u_1\big|_{t=0}=0,\ \ \Dt u_2+A^{1+\gamma}u_2=0,\ \ u_2\big|_{t=0}=u_0.
$$
Then, for the first equation, using the fact that $A^{1+\gamma}$ generates an analytic
 semigroup in $H$, we have the following maximal regularity result:
 $$
\| u_1\|_{C^{1+\kappa}(0,1; H^{-2\gamma})\cap C^\kappa(0,1; H^2)}\le
 C\|h\|_{C^\kappa(0,1;H^{-2\gamma})}.
$$
for all $0<\kappa<1$, see e.g., \cite{Cle87}. For the second component $u_2$, we
have a bit weaker estimate
$$
\| u_2\|_{C^1(0,1; H^{-2\gamma})\cap C(0,1; H^2)}\le
 C\|u_0\|_{H^2},
$$
see again \cite{Cle87}. Using now the interpolation
$$
\|u_2\|_{C^\kappa(0,1;H^1)}\le \|u_2\|_{C^{1}(0,1; H^{-2\gamma})\cap C(0,1; H^2)}
$$
for $0<\kappa\le\frac1{2(\gamma+1)}$, we get the desired result and finish the proof of the lemma.
 \end{proof}
 To apply this result to our case, we estimate the nonlinearity using the global
 Lipschitz continuity assumption:
 \begin{multline}
\|F(u)\|_{C^\kappa(0,1;H)}\le
C(1+\|u\|_{C^\kappa(0,1;H)})\le \eb\|u\|_{C^\kappa(0,1;H^1)}+\\+
C_\eb(1+\|u\|_{Lip(0,1;H^{-3\gamma-2})})\le
\eb\|u\|_{C^\kappa(0,1;H^1)}+\\+C_\eb(1+\|\Dt u\|_{L^\infty(0,1;H^{-3\gamma-2})})\le
\eb\|u\|_{C^\kappa(0,1;H^1)}+
C_\eb(1+\|u_0\|_{H^2}+\|g\|_H),
 \end{multline}
 where $\eb>0$ is arbitrary and we have used
 inequality \eqref{01.dis} in order to estimate the $H^{-3\gamma-2}$-norm of $\Dt u$.
\par
Applying estimate \eqref{01.mreg-g} to equation \eqref{01.lin} with $h(t)=A^{\gamma}(g-F(u(t)))$ and
 fixing $\eb>0$ small enough, we finally
 arrive at
 $$
\|u\|_{C(0,1;H^2)}\le C(1+\|u_0\|_{H^2})
$$
which gives us the desired estimate \eqref{01.h2-dis} and finishes the proof of the proposition.
\end{proof}
In what follows we will also need smoothing estimates for differences of solutions which,
in particular, will
 allow us to show that the IMs constructed in the phase space $H^{-\gamma}$ will be simultaneously IMs in more regular spaces $H^s$.

\begin{proposition}\label{Prop01.d-sm} Let the assumptions of Proposition \ref{Prop01.well}
hold and let $u_1(t)$
and $u_2(t)$ be two solutions
 of problem \eqref{1.FCH}. Then, for every $\beta>0$, the following estimate holds:
\begin{equation}\label{01.dsm}
\|u_1(t)-u_2(t)\|_{H^{2-\beta}}\le C_\beta t^{-1}\|u_1(0)-u_2(0)\|_{H^{-\gamma}},\ \ t\in(0,1],
\end{equation}
where the constant $C_\beta$ is independent of the choice of $u_1$ and $u_2$.
\end{proposition}
\begin{proof} Let $v(t):=u_1(t)-u_2(t)$ and let $w(t)=tv(t)$. Then, the last function solves
\begin{equation}\label{01.ddif}
\Dt w+A^{1+\gamma}w=\tilde h(t):=-tA^\gamma(F(u_1(t))-F(u_2(t)))+ v(t),\ \ w\big|_{t=0}=0.
\end{equation}
We want to apply the analogue of estimate \eqref{01.mreg-g} with $\kappa=0$ to this equation. However,
as well-known, the maximal regularity estimate works perfectly in H\"older spaces, but fails in $C$, so we need
 to decrease the regularity exponent (from $2$ till $2-\beta$) in order to restore the validity,
  see, say, \cite{Cle87,Tri78} for more details. This gives us the following estimate
  \begin{multline}
\|w\|_{C(0,1;H^{2-\beta})}\le C_\beta\|\tilde h\|_{C(0,1;H^{-2\gamma})}\le
C_\beta L\|w\|_{C(0,1;H)}+\|v\|_{C(0,1;H^{-2\gamma})}\le\\\le
\eb\|w\|_{C(0,1;H^{2-\beta})}+C_\eb\|v\|_{C(0,1;H^{-\gamma})}.
  \end{multline}
Fixing $\eb=\frac12$ in this estimate and using \eqref{01.lip}, we get the desired estimate \eqref{01.dsm}
 and finish the proof of the proposition.
\end{proof}
\begin{remark}\label{Rem01.e-sm} The restriction that the smoothing exponent in \eqref{01.dsm}
 is restricted by $2-\beta<2$ is related with the fact that $F'(u)$ is a bounded
  operator from $H$ to $H$ only. If we know, in addition, that
  \begin{equation}\label{01.esm}
\|F'(u)\|_{\Cal L(H^{s_0},H^{s_0})}\le C,
  \end{equation}
  for some $0<s_0<2$, we may get the analogue of the smoothing property \eqref{01.dsm}, where
  $2-\beta$ is replaced by $2+s_0-\beta$ (with $t^{-1}$  replaced by $t^{-2}$). Indeed,
   to get this estimate we just  need to make one more step. Namely, when \eqref{01.dsm} is already
    obtained, we need to return to equation \eqref{01.ddif},  apply the  parabolic
    regularity theorem to it in the space $H^{s_0-2\gamma}$ and use \eqref{01.esm} to estimate the terms
     related with the nonlinearity.
\end{remark}
\section{Inertial Manifolds and cone property}\label{s2}

The aim of this section is to recall the basic facts about the Inertial Manifolds (IMs) adapted to
 our model equation \eqref{1.FCH}. We will consider here only the case where the IM is constructed over
 the spectral subspace $H_N=\operatorname{span}\{e_1,\cdots,e_N\}$ generated by first $N$ eigenvectors
  of the operator $A$. Here and below, we denote by $P_N: H\to H_{N,+}$ the orthoprojector defined by
  $$
  P_N u:=\sum_{n=1}^N(u,e_n)e_n
  $$
and $Q_N:=1-P_N$. It is not difficult to see that $P_N$ and $Q_N$ are orthoprojectors in $H^s$, $s\in\R$
 and generate a splitting
 $$
 H^s=H_{N,+}^s\oplus H_{N,-}^s,\ \ H_{N,+}^s=H_{N,+},\ \ H_{N,-}^s=Q_NH^s
 $$
 of the space $H^s$ into the orthogonal sum of two spectral subspaces. Of course,
  the dimension of $H_N$ is $N$.

\begin{definition} \label{Def02.IM} A sub-manifold $\Cal M\subset H^{-\gamma}$ of dimension $N$
 is called an Inertial Manifold for equation \eqref{1.FCH} if the following conditions are satisfied:
 \par
 1. $\Cal M$ is invariant with respect to the solution semigroup $S(t)$ generated by
 \eqref{1.FCH}: $S(t)\Cal M=\Cal M$;
 \par
 2. $\Cal M$ is a graph of a globally Lipschitz continuous
 function $\Bbb M: H_{N,+}^{-\gamma}\to H_{N,-}^{-\gamma}$:
 \begin{equation}\label{02.graph}
\Cal M=\{u_++u_-,\ \ u_-=\Bbb M(u_+),\ \ u_+\in H_{N,+}^{-\gamma}\}.
 \end{equation}
 We will say that the IM $\Cal M$ is $C^{1+\alpha}$-smooth if $\Bbb M$ is $C^{1+\alpha}$-smooth.
 \par
 3. The manifold $\Cal M$ possesses the exponential tracking (=asymptotic phase) property, namely,
 there exists a positive constant $\theta$ such that for any $u_0\in H^{-\gamma}$ there exists
  a "trace" $\bar u_0\in\Cal M$ such that
  \begin{equation}
  \|S(t)u_0-S(t)\bar u_0\|_{H^{-\gamma}}\le Ce^{-\theta t}\|u_0-\bar u_0\|_{H^{-\gamma}}
  \end{equation}
  for some positive $C$.
\end{definition}
\begin{remark} As known, the above stated properties of IMs are closely related with
 {\it normal-hyperbolicity}. Indeed, usually the manifold $\Cal M$ is not only Lipschitz continuous,
 but also is $C^{1+\alpha}$-smooth for some small positive $\alpha$, so we may speak about tangential
  and transversal directions.
\par
   Then, as a rule the exponent of attraction   in directions
  transversal to the manifold  ($\theta$) is not only positive,
   but also {\it larger} than the Lyapunov exponents in the tangential directions. This, in particular, gives
    us the robustness of the IM with respect to perturbations, see \cite{Fen72,M-PS88,RT96,KZ15,Z14} for more details.
\end{remark}
The existence of an IM is usually verified by checking the so-called invariant {\it cone} property. To state it
 in our situation we introduce the following quadratic form:
 \begin{equation}\label{02.defV}
 V(\xi)=V_N(\xi):=\|Q_N\xi\|^2_{H^{-\gamma}}-\|P_N\xi\|^2_{H^{-\gamma}},\ \ \xi\in H^{-\gamma}
 \end{equation}
 and define the associated cone in the phase space $H^{-\gamma}$:
 \begin{equation}\label{02.defc}
K^+:=\bigg\{\xi\in H^{-\gamma}, V(\xi)\le0\bigg\}.
 \end{equation}

\begin{definition}\label{Def02.cs} Let the above assumptions hold.
We say that the solution semigroup $S(t)$ generated by equation \eqref{1.FCH} possesses the cone
 property (invariance of the cone $K^+$) if
\begin{equation}\label{02.cp}
\xi_1 - \xi_2 \in K^+ \Rightarrow S(t)\xi_1 - S(t)\xi_2 \in K^+, \text{ for all } t\ge 0,
\end{equation}
where $\xi_1, \xi_2 \in H^{-\gamma}$.
\par
Analogously, we say that $S(t)$ possesses the squeezing property if
 there exist positive $\theta$ and $C$ such that
\begin{equation}\label{02.sqz}
S(T)\xi_1- S(T)\xi_2 \not \in K^+ \Rightarrow \|S(t)\xi_1 - S(t)\xi_2\|_{H^{-\gamma}}
 \le C e^{-\theta t} \|\xi_1 - \xi_2\|_{H^{-\gamma}},\  t\in [0, T].
\end{equation}
\end{definition}
The key result of the theory of invariant manifolds is that (at least on the level of
 abstract semi-linear parabolic equations) the cone and squeezing properties imply the existence of an IM.

 \begin{theorem}\label{Th02.IM} Let the solution semigroup $S(t)$ of problem \eqref{1.FCH} possess
 the cone and
 squeezing properties. Then there exists a Lipschitz IM for this problem in the phase space $H^{-\gamma}$.
 \end{theorem}
The proof of this theorem can be found, e.g., in \cite{M-PS88,Z14}.
\par
 We just mention that the desired
 Lipschitz function $\Bbb M: H^{-\gamma}_{N,+}\to H^{-\gamma}_{N,-}$ can be obtained as
  follows: for a given $u_+\in H_{N,+}^{-\gamma}$ and $T>0$, one finds a unique
  solution $u=u_{T,u_+}(t)$ of the boundary value problem
\begin{equation}\label{01.constr}
\Dt A^{-\gamma}u+Au+F(u)=g,\ \ P_N u\big|_{t=0}=u_+,\ \ Q_Nu\big|_{t=-T}=0.
\end{equation}
Then, at the next step one passes to the limit $T\to\infty$ and find a backward
trajectory $u_{u_+}(t)$, $t\le0$:
\begin{equation}
u_{u_+}(t):=\lim_{T\to\infty} u_{T,u_+}(t).
\end{equation}
The existence of this limit is guaranteed by the squeezing property, see \cite{Z14} for details.
 Finally we define
\begin{equation}
\Bbb M(u_+):=Q_N u_{u_+}(0).
\end{equation}
Then the cone property guarantees us the Lipschitz continuity of $\Bbb M$ and the squeezing property implies in
 a standard way the exponential tracking property, see \cite{Z14} for more details. We also mention
  that the semigroup $S(t)$ restricted to the IM $\Cal M$ can be extended to a globally
   Lipschitz continuous group
  \begin{equation}
\|S(-t)\xi_1-S(-t)\xi_2\|_{H^{-\gamma}}\le Ce^{K|t|}\|\xi_1-\xi_2\|_{H^{-\gamma}},\ \ \xi_1,\xi_2\in\Cal M.
  \end{equation}
  This estimate follows from the fact that any trajectory $u(t)\in\Cal M$ has
   a structure $u(t)=u_+(t)+\Bbb M(u_+(t))$, where
  the function $u_+(t)\in H_{N,+}$ solves a system of ODEs
\begin{equation}\label{02.IF}
  \Dt u_++A^{1+\gamma}u_++A^{\gamma}P_N F(u_++\Bbb M(u_+))=P_NA^\gamma g
\end{equation}
with globally Lipschitz continuous nonlinearity. This system of ODEs is usually referred as
 an {\it Inertial Form} (IF) associated with equation \eqref{1.FCH} and gives us the desired
  finite-dimensional reduction constructed via IMs.

  \begin{corollary}\label{Cor02.H} Let the solution semigroup $S(t)$ of equation \eqref{1.FCH} satisfy
   the cone and squeezing properties in the phase space $H^{-\gamma}$. Then the IM $\Cal M$ in the space
    $H^{-\gamma}$ constructed in Theorem \ref{Th02.IM} is simultaneously an IM for equation \eqref{1.FCH}
     in any phase space $H^s$, $-\gamma\le s<2$.
  \end{corollary}
  Indeed, this statement is an immediate corollary of the construction of an IM for $H^{-\gamma}$ described in Theorem \ref{Th02.IM}
   and the smoothing property \eqref{01.dsm}.

\begin{remark}\label{Rem02.space} The result of Corollary \ref{Cor02.H} shows that the choice of the phase
 space where to verify the cone and squeezing properties is in our disposal and it is natural to
 fix this phase space in the way which simplifies  calculations. In particular, there are no connections
  between the initial problem {\it before} the cut-off of the nonlinearities making them
   globally Lipschitz and the technical choice of the phase space for proving the IM existence. Since
    the most delicate procedure in our proof is related with spatial averaging, we fix the
     $H^{-\gamma}$ as a phase space just in order to be able to treat the spatial averaging in
      the most convenient space $H$.
\end{remark}
We now discuss the ways to verify the above introduced cone and squeezing properties for
 equation \eqref{1.FCH}. To this end we introduce, following \cite{KZ15} the so called strong cone
  property in a differential form which allows us to verify cone and squeezing properties
  simultaneously and also gives normal hyperbolicity of the IM and its extra smoothness if
   $F(u)$ is smooth enough.

\begin{definition}\label{Def02.sc} Assume in addition that the function $F:H\to H$ is Gateaux
differentiable at every point $u\in H$ and its Gateaux derivative $F'(u)$ is a linear continuous operator
 in $H$. Then, the equation of variations
 \begin{equation}\label{1.lin}
\Dt A^{-\gamma}v+Av+l(t)v=0,\ \ l(t):=F'(u(t)),
 \end{equation}
 where $u(t):=S(t)u_0$
 which corresponds to equation \eqref{1.FCH} is well-defined. Clearly,
 \begin{equation}\label{1.Lip}
\|F'(u)\|_{\Cal L(H,H)}\le L
 \end{equation}
We say that equation \eqref{1.FCH} satisfies the strong cone property in a differential form if there are
Borel  measurable bounded function $\alpha: H\to\R$ and a positive constant $\mu$ such that
 $$
 0<\alpha_1\le \alpha(u)\le\alpha_2
 $$
 and
\begin{equation}\label{1.cone}
 \frac12\frac{d}{dt}V(v(t))+\alpha(u(t))V(v(t))\le-\mu\|v(t)\|_{H}^2
\end{equation}
for any $u_0\in H^{-\gamma}$ and any solution
 $v(t)$ of problem \eqref{1.lin} starting from $v_0\in H^{-\gamma}$.
\end{definition}
The next theorem is a key point in our method of constructing the IMs.

\begin{theorem}\label{Th02.IM-cone} Let the assumptions of Proposition \ref{Prop01.well}
 be satisfied and let, in addition,
equation \eqref{1.FCH} possess a strong cone property in
 a differential form for some $N\in\Bbb N$.
Then, the
 solution semigroup $S(t)$ possesses a cone and squeezing
 properties and, according to Theorem \ref{Th02.IM}
  also possesses an IM with the base $H_{N,+}=H_{N,+}^{-\gamma}$.
\end{theorem}
The proof of this result is given in \cite{KZ15}.
\par
Thus, in order to prove the existence of an IM for our equation \eqref{1.FCH}, it is sufficient to verify
only estimate \eqref{1.cone} for the linearized equation \eqref{1.lin}.
\par
The next result gives the extra smoothness of the constructed IM.

\begin{theorem}\label{Th02.IMsm} Let the assumptions of Theorem \ref{Th02.IM-cone} hold and
let, in addition, the
 nonlinearity $F$ satisfy
\begin{equation}\label{02.fdif}
\|F(u_1)-F(u_2)-F'(u_1)(u_1-u_2)\|_H\le C\|u_1-u_2\|_H\|u_1-u_2\|_{H^{2-\kappa}}^\delta,\
\ u_1,u_2\in H^{2-\kappa}
\end{equation}
for some small positive constants $\delta$ and $\kappa$. Then, the associated IM is $C^{1+\delta}$-smooth.
\end{theorem}
The proof of this theorem is given in \cite{KZ15} for the case $\gamma=1$,
but the case of general $\gamma$ is
 completely analogous.
\begin{remark}\label{Rem02.sm} We emphasize that the theorem gives $C^{1+\delta}$-smoothness of the IM for
 {\it small} positive $\delta$ only no
 matter how smooth the nonlinearity $F$ is. The space $H^{2-\kappa}$ in \eqref{02.fdif} is related only
 with the fact that in general we have parabolic smoothing property \eqref{01.dsm} for the exponents
  less than $2$. If we
  somehow know, in addition, that this smoothing property holds for the space $H^s$ with $s>2$,
   then $H^{2-\kappa}$ in \eqref{02.fdif}
   can be replaced by $H^s$. For instance, if \eqref{01.esm} is satisfied, $H^{2-\kappa}$
   can be replaced by $H^{s}$ with $s<s_0+2$. We also mention that estimate \eqref{02.fdif} is actually
   used only for $u_1,u_2\in\Cal M$, so we may check  it only under the extra assumption that
   $$
   \|Q_Nu_1\|_{H^{2-\kappa}}+\|Q_Nu_2\|_{H^{2-\kappa}}\le C
   $$
   for some $\kappa>0$ and sufficiently large $C$. Moreover, the estimate \eqref{01.esm} should
   also be checked for $\|u\|_{H^{2-\kappa}}\le C$ only.
\end{remark}

\section{Verification of the cone property via spatial averaging}\label{s3}
This section is devoted to verifying the strong cone condition \eqref{1.cone} for the solutions $v(t)$ of
 \eqref{1.lin}. We start with the simplest case where the so-called spectral gap conditions are satisfied.

\begin{proposition}\label{Prop1.SG} Let $N\in\Bbb N$ be such that
\begin{equation}\label{1.SG}
\frac{\lambda^{1+\gamma}_{N+1}-\lambda_N^{1+\gamma}}{\lambda_{N+1}^\gamma+\lambda_N^\gamma}>L.
\end{equation}
Then the corresponding equation \eqref{1.lin} possesses the strong cone property \eqref{1.cone} with
\begin{equation}
\alpha:=\lambda_N^{1+\gamma}\frac{\lambda_{N+1}^\gamma}{\lambda_N^\gamma+\lambda_{N+1}^\gamma}+
\lambda_{N+1}^{1+\gamma}\frac{\lambda_N^\gamma}{\lambda_N^\gamma+\lambda_{N+1}^\gamma},\ \
 \mu:=\frac{\lambda^{1+\gamma}_{N+1}-\lambda_N^{1+\gamma}}{\lambda_{N+1}^\gamma+\lambda_N^\gamma}-L.
\end{equation}
\end{proposition}
\begin{proof} Multiplying equation \eqref{1.lin} by $Q_Nv-P_Nv$, we get
\begin{multline}\label{1.dif}
\frac12\frac d{dt}V(v(t))+\alpha V(v(t))+
((\alpha A^{-\gamma}-A)P_Nv,P_Nv)+\\+
((A-\alpha A^{-\gamma})Q_Nv,Q_Nv)=-(l(t)v,Q_Nv-P_Nv).
\end{multline}
Using the fact that the function $x\to x-\alpha x^{-\gamma}$ is monotone increasing, we can estimate
$$
((\alpha A^{-\gamma}-A)P_Nv,P_Nv)=
\sum_{n=1}^N(\alpha\lambda_n^{-\gamma}-\lambda_n)|v_n|^2\ge
\sum_{n=1}^N(\alpha\lambda_N^{-\gamma}-\lambda_N)|v_n|^2=(\alpha\lambda_N^{-\gamma}-\lambda_N)\|P_Nv\|^2_H
$$
and, analogously,
$$
((A-\alpha A^{-\gamma})Q_Nv,Q_Nv)=
\sum_{n=N+1}^\infty(\lambda_n-\alpha\lambda_n^{-\gamma})|v_n|^2\ge
(\lambda_{N+1}-\alpha\lambda_{N+1}^{-\gamma})\|Q_Nv\|^2_H.
$$
Since, by our choice the exponent $\alpha$ solves
$$
\lambda_{N+1}-\alpha\lambda_{N+1}^{-\gamma}=\alpha\lambda_N^{-\gamma}-\lambda_N,
$$
and elementary calculation shows that
$$
((\alpha A^{-\gamma}-A)P_Nv,P_Nv)+
((A-\alpha A^{-\gamma})Q_Nv,Q_Nv)\ge
\frac{\lambda_{N+1}^{1+\gamma}-\lambda_N^{1+\gamma}}{\lambda_{N+1}^\gamma+\lambda_N^\gamma}\|v\|^2_H.
$$
Finally, the Cauchy-Schwarz inequality together with assumption \eqref{1.Lip} gives
$$
|(l(t)v,Q_Nv-P_Nv)|\le L\|v\|^2_H
$$
and inserting the obtained estimates to \eqref{1.dif} we arrive at \eqref{1.cone} and
 finish the proof of the proposition.
\end{proof}
The rest of this section is devoted to the case when the spectral gap condition \eqref{1.SG}
 is not satisfied, but instead the nonlinearity satisfies the so-called spatial averaging principle.
 To state this principle, we introduce for every $k\in\Bbb N$ the following orthoprojectors:
 $$
\Cal P_{k,N}u:=\sum_{j:\ \lambda_j<\lambda_N-k}(u,e_j)e_j,\ \
\Cal Q_{k,N}u:=\sum_{j:\ \lambda_j>\lambda_N+k}(u,e_j)e_j,\ \Cal I_{k,N}:=1-\Cal P_{k,N}-\Cal Q_{k,N}
$$
Thus, instead of splitting $v=v_++v_-$ on lower ($v_+:=P_Nv$) and higher ($v_-=Q_Nv$) modes, we now use
 the splitting
$$
v=v_{++}+v_I+v_{--},\ v_{++}:=\Cal P_{k,N}v,\ v_I:=\Cal I_{k,N}v,\ \ v_{--}:=\Cal Q_{k,N}v
$$
on essentially lower, essentially higher and {\it intermediate} modes. The key assumption in the spatial
 averaging method is that
the operator $F'(u(t))$ restricted to the intermediate modes is close to the scalar operator. Then, we say
 that $F$ satisfies the spatial averaging principle if there exists $\theta>0$ such that, for
  every positive $\delta<L$ and natural number  $k$ there exist
  infinitely many values of $N\in\Bbb N$ such that
\begin{equation}\label{1.spa}
  \|\Cal I_{k,N}F'(u)\Cal I_{k,N}-a(u)\Cal I_{k,N}\|_{\Cal L(H,H)}\le \delta
\end{equation}
uniformly with respect to $u\in H$ and $\lambda_{N+1}-\lambda_N\ge\theta$.
Here $a:H\to\R$ may depend on $\delta$ and~$N$.
\par
We are now ready to state and prove the main result of this section.

\begin{theorem}\label{Th1.main} Let the nonlinearity satisfy the spatial averaging principle and let the
 involving constants $\theta$, $k$, $L$ and $N$ satisfy
\begin{equation}
\frac{\theta}8-\delta-\gamma 2^{\gamma+1} L\frac{k}{\lambda_N-k}>0,\ \ \frac12 k-\frac{8L^2}{\theta}-2L\ge0
\end{equation}
and, in addition,  $\lambda_N>L$ and $k\le\lambda_N/2$.
\par
Then equation \eqref{1.FCH} possesses a strong cone property in the form of \eqref{1.cone}.
\end{theorem}
\begin{proof} We just need to estimate the terms in \eqref{1.dif} in a more accurate way. Namely, for lower modes
$\Cal P_{k,N}v$, we will have
\begin{multline}\label{1.start}
((\alpha A^{-\gamma}-A)\Cal P_{k,N}v,\Cal P_{k,N}v)=
\sum_{n: \lambda_n<\lambda_N-k}
\(\frac{\lambda_N^\gamma}{\lambda_n^\gamma}\cdot
\frac{\lambda_{N+1}^\gamma(\lambda_{N+1}+\lambda_N)}{\lambda_N^\gamma+\lambda_{N+1}^\gamma}-
\lambda_n\)|v_n|^2\ge\\\ge \sum_{n: \lambda_n<\lambda_N-k}
\(
\frac{\lambda_{N+1}^\gamma(\lambda_{N+1}+\lambda_N)}{\lambda_N^\gamma+\lambda_{N+1}^\gamma}-
\lambda_N+k\)|v_n|^2=(\bar\mu+k)\|\Cal P_{k,N}v\|^2_H,
\end{multline}
where $\bar\mu:=\frac{\lambda_{N+1}^{1+\gamma}-\lambda_N^{1+\gamma}}
{\lambda_{N+1}^\gamma+\lambda_N^\gamma}$. Arguing analogously, we also get
$$
((A-\alpha A^{-\gamma})\Cal Q_{k,N}v,\Cal Q_{k,N}v)\ge (\bar\mu+k)\|\Cal Q_{k,N}v\|^2_H.
$$
In addition, we need the analogue of \eqref{1.start} for the $H^{-\gamma}$-norm. Namely,
\begin{multline}\label{1.starg}
((\alpha A^{-\gamma}-A)\Cal P_{k,N}v,\Cal P_{k,N}v)=
\sum_{n: \lambda_n<\lambda_N-k}
\(\lambda_N^\gamma
\frac{\lambda_{N+1}^\gamma(\lambda_{N+1}+\lambda_N)}{\lambda_N^\gamma+\lambda_{N+1}^\gamma}-
\lambda_n\lambda_n^\gamma\)(\lambda_n^{-\gamma}|v_n|^2)\ge\\\ge \sum_{n: \lambda_n<\lambda_N-k}
\lambda_N^\gamma\(
\frac{\lambda_{N+1}^\gamma(\lambda_{N+1}+\lambda_N)}{\lambda_N^\gamma+\lambda_{N+1}^\gamma}-
\lambda_N+k\)\lambda_n^{-\gamma}|v_n|^2=\lambda_N^{\gamma}(\bar\mu+k)\|\Cal P_{k,N}v\|^2_{H^{-\gamma}}.
\end{multline}
Moreover, estimating lower-intermediate and higher-intermediate modes exactly as in Proposition \ref{Prop1.SG}
and using that
$$
((\alpha A^{-\gamma}-A)P_Nv,P_Nv)=((\alpha A^{-\gamma}-A)\Cal P_{k,N}v,\Cal P_{k,N}v)+
((\alpha A^{-\gamma}-A)P_N\Cal I_{k,N}v,P_N\Cal I_{k,N}v)
$$
and the analogous expression for $Q_Nv$ component, we transform \eqref{1.dif} to
\begin{multline}\label{1.dif1}
\frac12\frac d{dt}V(v(t))+\alpha V(v(t))+\frac12\bar\mu\|v(t)\|^2_H+\\+
\frac12k(\|\Cal P_{k,N}v\|^2_H+\|Q_{k,N}v\|^2_H)+
\frac12\lambda_N^{\gamma}(\bar\mu+k)\|\Cal P_{k,N}v\|^2_{H^{-\gamma}}\le
 -(l(t)v, Q_Nv-P_Nv).
\end{multline}
To estimate the right-hand side we use that $1=\Cal P_{k,N}+\Cal I_{k,N}+\Cal Q_{k,N}$:
\begin{multline}
-(l(t)v,Q_Nv-P_Nv)=(l(t)v,\Cal P_{k,N}v)-(l(t)v,\Cal Q_{k,N}v)-
(l(t)v,Q_N\Cal I_{k,N}v-P_N\Cal I_{k,N}v)\le\\\le
L\|v\|_H(\|\Cal P_{k,N}v\|_H+\|\Cal Q_{k,N}v\|_H)-(l(t)v,Q_N\Cal I_{k,N}v-P_N\Cal I_{k,N}v)\le\\\le
 \frac{\bar\mu}8\|v\|^2_H+\frac{2L^2}{\bar\mu}(\|\Cal P_{k,N}v\|^2_H+\|\Cal Q_{k,N}v\|^2_H)-
 (l(t)v,Q_N\Cal I_{k,N}v-P_N\Cal I_{k,N}v).
\end{multline}
We may continue this estimate as follows
\begin{multline}
-(l(t)v,Q_N\Cal I_{k,N}v-P_N\Cal I_{k,N}v)=-(l(t)\Cal P_{k,N}v,Q_N\Cal I_{k,N}v-P_N\Cal I_{k,N}v)-\\-
(l(t)\Cal Q_{k,N}v,Q_N\Cal I_{k,N}v-P_N\Cal I_{k,N}v)-
(l(t)\Cal I_{k,N}v,Q_N\Cal I_{k,N}v-P_N\Cal I_{k,N}v)\le\\\le
L\|v\|_H(\|\Cal P_{k,N}v\|_H+\|\Cal Q_{k,N}\|_H)-(\Cal I_{k,N}l(t)\Cal I_{k,N}v,Q_Nv-P_Nv)\le\\\le
\frac{\bar\mu}8\|v\|^2_H+\frac{2L^2}{\bar\mu}(\|\Cal P_{k,N}v\|^2_H+\|\Cal Q_{k,N}v\|^2_H)-
(\Cal I_{k,N}l(t)\Cal I_{k,N}v,Q_Nv-P_Nv).
\end{multline}
Using now \eqref{1.spa}, we get
\begin{equation}\label{1.intr}
-(\Cal I_{k,N}l(t)\Cal I_{k,N}v,Q_Nv-P_Nv)\le
 -a(u(t))(\|Q_N\Cal I_{k,N}v\|^2_H-\|P_N\Cal I_{k,N}v\|^2_H)+\delta\|v\|^2_H.
\end{equation}
To transform the right-hand side of this inequality, we need
 the following straightforward estimates
\begin{multline}
\bigg|\lambda_N^\gamma\|P_N\Cal I_{k,N}v\|_{H^{-\gamma}}^2-\|P_N\Cal I_{k,N} v\|_H^2\bigg|\le\\\le
 \sum_{n:\, \lambda_N - k \le \lambda_n \le\lambda_N}
 |\lambda_N^\gamma-\lambda_n^\gamma|\lambda_n^{-\gamma}|v_n|^2\le
   \frac{(\lambda_N^\gamma-(\lambda_N-k)^\gamma)}{(\lambda_N-k)^\gamma}\|P_N\Cal I_{k,N}v\|_{H}^2
\end{multline}
and
\begin{multline}
\bigg|\lambda_N^\gamma\|Q_N\Cal I_{k,N}v\|_{H^{-\gamma}}^2-\|Q_N\Cal I_{k,N} v\|_H^2\bigg|\le\\
\le \sum_{n:\, \lambda_{N+1} \le \lambda_n \le \lambda_{N} + k}
|\lambda_N^\gamma-\lambda_n^\gamma|\lambda_n^{-\gamma}||v_n|^2\le
 \frac{(\lambda_N+k)^\gamma-\lambda_N^\gamma}{(\lambda_N+k)^\gamma}\|Q_N\Cal I_{k,N} v\|^2_H.
\end{multline}
Moreover, as not difficult to check,
$$
\frac{(a+x)^\gamma-a^\gamma}{(a+x)^\gamma}\le \frac{a^\gamma-(a-x)^\gamma}{(a-x)^\gamma},\ \ 0<x\le a.
$$
Therefore
\begin{multline}
-a(u(t))(\|Q_N\Cal I_{k,N}v\|^2_H-\|P_N\Cal I_{k,N}v\|^2_H)\le\\\le -a(u(t))\lambda_N^\gamma
(\|Q_N\Cal I_{k,N}v\|^2_{H^{-\gamma}}-\|P_N\Cal I_{k,N}v\|^2_{H^{-\gamma}})+
2L\frac{\lambda_N^\gamma-(\lambda_N-k)^\gamma}{(\lambda_N-k)^\gamma}\|v\|^2_H.
\end{multline}
Finally, we estimate the first-term in the right-hand side through the function $V(v(t))$ as follows:
\begin{multline}
-a(u(t))\lambda_N^\gamma
(\|Q_N\Cal I_{k,N}v\|^2_{H^{-\gamma}}-\|P_N\Cal I_{k,N}v\|^2_{H^{-\gamma}})=\\=
-a(u(t))\lambda_N^\gamma V(v(t))+
a(u(t))\lambda_N^\gamma(\|\Cal P_{k,N}v\|^2_{H^{-\gamma}}-\|\Cal Q_{k,N}v\|^2_{H^{-\gamma}})\le\\\le
-a(u(t))\lambda_N^\gamma V(v(t))+2L\lambda_N^\gamma\|\Cal P_{k,N}v\|^2_{H^{-\gamma}}+
2L\|\Cal Q_{k,N}v\|^2_{H}.
\end{multline}
Combining the obtained estimates, we get
\begin{multline}\label{1.fest}
-(l(t)v,Q_Nv-P_Nv)\le -a(u(t))\lambda_N^\gamma V(v(t))+
 \(\frac{\bar\mu}4+\delta+2L\(\(1+\frac{k}{\lambda_N-k}\)^\gamma-1\)\)\|v\|^2_H+\\+
\(\frac{4L^2}{\bar\mu}+2L\)(\|\Cal P_{k,N}v\|^2_H+\|\Cal Q_{k,N}v\|^2_H)+
2L\lambda_N^\gamma\|\Cal P_{k,N}v\|^2_{H^{-\gamma}}.
\end{multline}
Using the elementary inequality
$$
(1+x)^\gamma-1\le \gamma2^\gamma x,\ \ x\in(0,1)
$$
and inserting \eqref{1.fest} into \eqref{1.dif1}, we get the desired inequality
\begin{multline}\label{1.fin}
\frac12\frac d{dt}V(v(t))+(\alpha+a(u(t))\lambda_N^\gamma)V(v(t))+\\+
\(\frac{\bar\mu}4-\delta-\gamma 2^{\gamma+1} L\frac{k}{\lambda_N-k}\)\|v\|^2_H+\\+
\(\frac12 k-\frac{4L^2}{\bar\mu}-2L\)\(\|\Cal P_{k,N}v\|^2_H+\|\Cal Q_{k,N}v\|^2_H\)+
\lambda_N^\gamma(\frac k2-2L)\|\Cal P_{k,N}v\|^2_{H^{-\gamma}}\le0.
\end{multline}
Using the obvious inequality
$$
\frac{|x^{1+\gamma}-y^{1+\gamma}|}{x^\gamma+y^\gamma}\ge \frac12|x-y|,\ \ x,y\ge0
$$
and the assumptions of the theorem, we see that \eqref{1.fin} implies the desired cone property and finishes the proof of the theorem.
\end{proof}

\section{The truncation procedure}\label{s4}
Note that in the previous sections, we have assumed that the nonlinearity $F(u)$ is {\it globally}
Lipschitz continuous and satisfies the spatial averaging principle also uniformly with
 respect to $u\in H$. These assumptions look very restrictive since in applications we usually have {\it growing}
  nonlinearities. The standard strategy here is to verify first the existence of an absorbing ball in
   some higher order space $H^s$ and then cut-off the nonlinearity outside of this
    ball, see \cite{FST88,M-PS88,Z14}.
   However this truncation is rather delicate when the spatial averaging is involved since we should
    preserve spatial averaging structure under this truncation. For the case of scalar
     reaction-diffusion equation, the proper cut-off procedure has been suggested in \cite{M-PS88} and
      alternative construction which is well-adapted for the case when the average  $a(u)$ of
      the nonlinearity $f'(u)$
       is identically zero has been introduced in \cite{K18}. In this section, we present a
       combination of two
       above mentioned methods which will allow us to treat both cases from the unified point of view.
       \par
       Let $\phi\in C^\infty(R)$ be such that $\phi(z)=z$ for $|z|\le 1$ and $\phi(z)=2$ for $|z|\ge 2$.
        Then for a
       given positive constant $C_*$ and sufficiently large exponent $s$, we define the function
       $W: H\to H$ via
       \begin{equation}\label{5.W}
        W(u)=\sum_{n=1}^\infty C_*\lambda_n^{-s/2}\phi\(\frac{\lambda_n^{s/2}(u,e_n)}{C_*}\)e_n.
       \end{equation}
The elementary properties of this truncation function are collected in the following proposition.
\begin{proposition}\label{Prop5.W} Let the function $W$ be defined via \eqref{5.W}. Then,
\par
1. The map $W$ is bounded and continuous as a map from $H$ to $H^{s_0}$, where $s_0>0$ is such that
\begin{equation}\label{5.sm-eat}
\sum_{n=1}^\infty\lambda_n^{s_0-s}<\infty.
\end{equation}
\par
2. $W(u)\equiv u$ if $u\in H^s$ and $\|u\|_{H^s}\le C_*$.
\par
3. The function $W$
 is Hadamard differentiable as a map from $H$ to $H$ and the derivative is given by
\begin{equation}\label{5.Wdif}
W'(u)v=\sum_{n=1}^\infty \phi'\(\lambda_n^{s/2}\frac{(u,e_n)}{C_*}\)(v,e_n)e_n.
\end{equation}
4. There exists a positive constant $C$ such that, for every $\kappa\in \R$
\begin{equation}\label{5.W-sm}
\|W'(u)\|_{\Cal L(H^\kappa,H^\kappa)}\le C,\  \ \|W'(u_1)-W'(u_2)\|_{\Cal L(H^\kappa,H^\kappa)}
\le C\|u_1-u_2\|_{H^{s}}
\end{equation}
for all $u,u_1,u_2\in H^s$.
\end{proposition}
\begin{proof} The first statement is straightforward. Indeed, let $u,v\in H$. Then,
due to \eqref{5.sm-eat} and boundedness of $\phi$, for every $\eb>0$, there exists $M=M(\eb)$ such that
$$
\|W(u+v)-W(u)\|_{H^{s_0}}^2\le\frac{\eb^2}2+\sum_{n=1}^MC_*^2\lambda_n^{s_0-s}
\(\phi\(\frac{\lambda_n^{s/2}(u+v,e_n)}{C_*}\)-\phi\(\frac{\lambda_n^{s/2}(u,e_n)}{C_*}\)\)^2.
$$
Since the sum in the RHS has now only finitely many terms and $\phi$ is continuous, we may make the sum less
than $\frac{\eb^2}2$ by taking the $H$-norm of $v$ small enough. This proves the continuity.
\par
 To verify the second property, let us take $u\in H^{s}$ such that
$$
\|u\|^2_{H^s}:=\sum_{n=1}^\infty \lambda_n^{s}(u,e_n)^2\le C_*^2,
$$
then $|(u,e_n)|\le C_*\lambda_n^{-s/2}$ and therefore $\phi\(\frac{\lambda_n^{s/2}(u,e_n)}{C_*}\)=
\frac{\lambda_n^{s/2}(u,e_n)}{C_*}$ and
$W(u)=u$.
\par
Let us verify the differentiability. To this end, we need to estimate
$$
\|W(u+th)-W(u)-tW'(u)h\|_H^2=\!
\sum_{n=1}^\infty\(\int_0^1\phi'\!\(\frac{\lambda_n^{\frac s2}(u+tlh,e_n)}{C_*}\)\!-
\phi'\!\(\frac{\lambda_n^{\frac s2}(u,e_n)}{C_*}\)dl\)^2\!\!\! t^2(h,e_n)^2.
$$
To check the Hadamard differentiability, we need to take $h\in \Cal K$ where $\Cal K$ is
 a compact set in $H$, so we have uniform smallness of the tails $\sum_{n=M}^\infty(h,e_n)^2$. Using also
 that $\phi'$ is bounded, for every $\eb>0$, we may find $M=M(\eb)$ such that
 $$
 \sum_{n=M}^\infty\(\int_0^1\phi'\(\frac{\lambda_n^{s/2}(u+tlh,e_n)}{C_*}\)-
\phi'\(\frac{\lambda_n^{s/2}(u,e_n)}{C_*}\)dl\)^2 t^2(h,e_n)^2\le \frac\eb2 t^2
$$
for all $t\in[0,1]$ and all $h\in\Cal K$. Passing to the limit $t\to0$ in the remaining {\it finite}
sum
$$
 \sum_{n=1}^M\(\int_0^1\phi'\(\frac{\lambda_n^{s/2}(u+tlh,e_n)}{C_*}\)-
\phi'\(\frac{\lambda_n^{s/2}(u,e_n)}{C_*}\)dl\)^2 t^2(h,e_n)^2
$$
is immediate since $\phi'$ is smooth, so we may make it less than $\frac\eb2 t^2$ by
 taking $t$ small enough. This proves the differentiability.
\par
Finally, the 4th property is an immediate corollary of the estimate
$$
\left|\phi'\(\frac{\lambda_n^{s/2}(u_1,e_n)}{C_*}\)-\phi'\(\frac{\lambda_n^{s/2}(u_2,e_n)}{C_*}\)\right|^2\le
C\lambda_n^{s}|(u_1-u_2,e_n)|^2\le C\|u_1-u_2\|^2_{H^s}
$$
and the fact that $\phi'$ is uniformly bounded. Thus, the proposition is proved.
\end{proof}
We now turn to more general semi-linear parabolic equation
\begin{equation}\label{5.maineq}
A^{-\gamma}\Dt u+Au+f(u)=g,\ \ u\big|_{t=0}=u_0,
\end{equation}
where $g\in H$ and $f$ is a given nonlinearity which is no more assumed to be globally
 bounded or/and globally Lipschitz. Instead, we assume that this problem is well-posed in
 a phase space $H^{s'}$ for some $s'\in\R$ and generates
  a dissipative semigroup $\bar S(t):H^{s'}\to H^{s'}$ there. We also assume that the ball $\Cal B$
   of radius $C_*$ in the space $H^s$, for some $s> \max\{s',-\gamma\}$ is a
   (semi)invariant absorbing ball for the semigroup $\bar S(t)$. The latter means that
   \par
   1. $\bar S(t)\Cal B\subset\Cal B$;
   \par
   2. For every bounded set $B\subset H^{s'}$, there exists $T=T(B)$ such that
   $$
   \bar S(t)B\subset\Cal B,\ \text{ if }\  t\ge T.
   $$
Roughly speaking, the idea is to define the truncated nonlinearity as $F(u):=f(W(u))$. Then, we will have
$$
F(u)=f(u),\ \ u\in\Cal B,
$$
but in order to verify that $F$ satisfies the conditions of Theorem \ref{Th1.main}, we need some restrictions on the map $f$.
Namely,
\par
{\bf Assumption I.} The map $f: H^{s_0}\to H$ is continuous and is locally bounded. Here
 $s_0=s_0(s)$ is the same as in Proposition \ref{Prop5.W}. The map $a:H^{s_0}\to\R$
 is also continuous and locally bounded.
\par
{\bf Assumption II.}
\par
{\bf a)} The map $f:H^{s_0}\to H$ is Gateaux differentiable at any point $u\in H^{s_0}$ and its
 derivative $f'(u)$
 is linear and can be extended to the linear continuous operator in $H$:
  $f'(u)\in\Cal L(H,H)$ for any $u\in H^{s_0}$. Moreover,
  \begin{equation}
f'\in C^\eb(H^{s_0},\Cal L(H,H))
  \end{equation}
  for some $\eb>0$.
\par
{\bf b)} The map $u\to a(u)$ is Gateaux differentiable as a map
   from $H^{s_0}$ to $\R$ and its derivative has a form $a'(u)v=(a'(u),v)$ where $a'(u)\in H$. Moreover,
   $$
   a'\in C^\eb(H^{s_0},H)
   $$
   for some $\eb>0$.
   \par
   {\bf c)} The map $f'$ is well-defined and is locally bounded as a map from
   $H^{s_0}$ to $\Cal L(H^{s_0},H^{s_0})$.
\par
{\bf Assumption III.} The following version of spatial averaging principle is satisfied:
there exists a function $a:H^{s_0}\to \R$ such that,
for every bounded set $B\subset H^{s_0}$
and every $\delta>0$ and $k>0$, there exists an infinite sequence of $N\in\Bbb N$ such that
\begin{equation}\label{5.spa}
\sup_{u\in B}\|\Cal I_{k,N} f'(u)\Cal I_{k,N}v-a(u)\Cal I_{k,N}v\|_H\le\delta\|v\|_H,\ \ \forall v\in H,
\end{equation}
compare with \eqref{1.spa}.
\par
We start with the simplest case where the spatial average $a(u)$ vanishes identically.

\begin{theorem}\label{Th5.a0} Let the nonlinearity $f$ satisfy Assumptions I,II and III and let $a(u)\equiv0$. Then, the
 truncated nonlinearity
 \begin{equation}\label{5.naive}
F(u):=f(W(u))
 \end{equation}
 satisfies the assumptions of Theorem \ref{Th1.main} and therefore, the associated truncated
  equation \eqref{1.FCH} possesses a family of IMs $\Cal M=\Cal M_N$ for infinitely many values of $N$.
\end{theorem}
\begin{proof} Indeed, according to Assumption I and the first statement of Proposition \ref{Prop5.W}, the
 map $F:H\to H$ is globally bounded and continuous. From Assumption II and the third statement
  of Proposition \ref{Prop5.W},
 we conclude that the map $F$ is Gateaux (and even Hadamard) differentiable
 and the following chain rule formula holds:
\begin{equation}\label{5.gooddif}
 F'(u)=f'(W(u))W'(u).
\end{equation}
Indeed, let $u,v\in H$ and $t\ge0$. Then
\begin{multline}\label{5.int-inc}
\|f(W(u+tv))-f(W(u))-tf'(W(u))W'(u)v\|_H\le\\\le
\|\int_0^1[f'(W(u)+\kappa(W(u+tv)-W(u)))-f'(W(u))]\,d\kappa\(W(u+tv)-W(u)\)\|_H+\\+
\|f'(W(u))(W(u+tv)-W(u)-tW'(u)v)\|_H\le C\|W(u+tv)-W(u)-tW'(u)v\|_H+\\+
 C\|W(u+tv)-W(u)\|_{H^{s_0}}^\eb\|W(u+tv)-W(u)\|_H
\end{multline}
and, since $W$ is Gateaux differentiable as a map from $H$ to $H$ and is continuous as a
 map from $H$ to $H^{s_0}$, we see that the right-hand side is of order $o(t)$. This proves
  the differentiability and verifies \eqref{5.gooddif}.
  \par
In particular, \eqref{5.gooddif} shows that $F'(u)$ is globally bounded in $\Cal L(H,H)$, so \eqref{1.Lip}
 is satisfied for the properly chosen constant $L$.
\par
Finally, inserting $W'(u)v$ instead of $v$ in \eqref{5.spa} and using that $a\equiv0$ and that the
operator $W'(u)$ is diagonal in the base of eigenvectors of $A$
 (and consequently $\Cal I_{k,N}W'=W'\Cal I_{k,N}$), together with the  boundedness of $W'$,
  we get that the spatial averaging condition \eqref{1.spa} is
  also satisfied for infinitely many values of $N$s. This finishes the proof of the theorem.
\end{proof}
We now return to the non-truncated equation \eqref{5.maineq} and give (following \cite{KZ18}) the natural definition
 of the IM for non-scaled case.
 \begin{definition}\label{Def5.IMg} Let $\bar S(t):\Phi\to\Phi$ be a semigroup acting in a Banach space $\Phi$ and
possessing the invariant bounded absorbing set $\Cal B$ in it. Assume that
\par
1) There exists another Banach space $\Psi$ and a dissipative semigroup $S(t)$ in $\Psi$.
\par 	
2) There exists a bi-Lipschitz embedding $E: \Cal B\to \Psi$ such that
$$
S(t) = E\circ \bar S(t)\circ E^{-1}
$$
on $E(\Cal B)\subset\Psi$ 	.
\par
3) The dynamical system $S(t)$ possesses an IM $\Cal M$ in the phase space $\Psi$.
\par
Then $\Cal M$ is referred as a (generalized) inertial manifold for the semigroup $\bar S(t)$ in the
 sense of Definition \ref{Def02.IM}.
This manifold is called $C^{1+\eb}$-smooth if both $E$ and $\Cal M$ are $C^{1+\eb}$-smooth.
 \end{definition}
 In our particular case $\Phi=H^{s'}$, $\Psi=H^{-\gamma}$, $\Cal B\subset \Phi\cap\Psi$ and
  the semigroups $\bar S(t)$ and $S(t)$ are the
 solution operators for equations \eqref{5.maineq} and \eqref{1.FCH} respectively
  and $E=\operatorname{Id}$. So, we have proved the following result.
  \begin{corollary} Let the assumptions of Theorem \ref{Th5.a0} hold and let, in addition, the solution
  semigroup $\bar S(t)$ possess an invariant absorbing ball $\Cal B$ in $H^{s}$. Then, there are
   infinitely many $N$s such that
  equation
  \eqref{5.maineq} possesses an IM of dimension $N$ in the sense of Definition \ref{Def5.IMg} and
  the associated truncated semigroup $S(t)$ is defined by equation \eqref{1.FCH}.
  \end{corollary}
\begin{remark}\label{Rem5.IM} The IM for the equation \eqref{1.FCH} with already truncated
 nonlinearity is usually unique if the dimension $N$ is fixed. However, the non-uniqueness
  of the IM for the
  initial non-truncated equations appears since there are many ways to make the cut-off procedure.
   Note also that the IM $\Cal M$ is strictly invariant for the truncated semigroup $S(t)$ only,
   and may be not invariant for the initial semigroup $\bar S(t)$. On the other hand, the manifold $\Cal M$
   always contains the image $E(\Cal A)$ of a global attractor $\Cal A$ of the initial equation, so
    it always generates an Inertial Form for the initial dynamics on the global attractor,
     see \cite{FST88,T97,M-PS88,KZ18,KZ17,Z14} for more details.
\end{remark}

We now return to the general case $a(u)\ne0$. In this case, the naive choice \eqref{5.naive} is
 no longer working (since for truncated nonlinearity we then have $a(W(u))W'(u)$ which is no
 more a scalar operator and everything crushes). So, we need to proceed in a more delicate way.
\par
Namely, following  \cite{M-PS88}, we assume that the absorbing ball $\Cal B$ is bounded in
 $H^2$ by the constant
 $R$ and introduce a cut-off function $\varphi(z)$ which equals to $0$ for $z\le R^2$
 and equals to $-1/2$ if
  $z\ge R_1^2$ for some $R_1>R$. Then, we define the map $T=T_N:H\to H$ via
\begin{equation}\label{5.T}
T(u):=\varphi(\|AP_Nu\|^2_H)AP_Nu.
\end{equation}
The key property of this map is stated in the following lemma.
\begin{lemma}\label{Lem5.T} It is possible to fix the cut-off function $\varphi$ in such a way that
\begin{equation}\label{5.T-good}
(T'(u)v,v)\le 0,\ \ v\in H
\end{equation}
and $(T'(u)v,v)=-\frac12\|P_Nv\|^2_{H^1}$ if $\|P_Nu\|_{H^2}\ge R_1$.
\end{lemma}
The proof of this lemma is given in \cite{M-PS88} (see also \cite{Z14}).
\par
We fix one more smooth cut-off function $\theta(z)$ which equals to one if $z\le \bar R^2$ and
 zero if $z>4\bar R^2$, where the parameter $\bar R$ is chosen in such
  a way that $\|\Cal B\|_H\le \bar R$ and define
\begin{equation}\label{5.comp}
F(u):=f(W(u))-a(W(u))W(u)+\theta(\|u\|^2_H)a(W(u))u+T_N(u).
\end{equation}
 Then, exactly as in the case, $a=0$, the function $F$ will be bounded and continuous as the map
  from $H$ to $H$ and its Gateaux derivative will have the form
  \begin{multline}\label{5.derder}
F'(u)v=[f'(W(u))W'(u)v-a(W(u))W'(u)v]+\theta(\|u\|^2_H)a(W(u))v-\\-(a'(W(u)),W'(u)v)W(u)+
[2\theta'(\|u\|^2_H)(u,v)a(W(u))+\\+\theta(\|u\|^2_H)(a'(W(u)),W'(u)v)]u+T'_N(u)v=\\=
l_1(u)v+l_2(u)v+l_3(u)v+l_4(u)v+T'_N(u)v.
  \end{multline}
  Indeed, the verification of \eqref{5.derder} is straightforward and is similar to what
   we did to check \eqref{5.gooddif}, so we left the details to the reader.
   \par
Note that only the term $T_N(u)$ depends explicitly on $N$ now, so the norms of all other terms
 are independent of $N$. In particular, since $Q_NT_N(u)\equiv 0$, we have that
 \begin{equation}\label{5.q-est}
\|Q_NF(u)\|_{H}\le C,\ \ u\in H,
 \end{equation}
 where $C$ is independent of $N$.

\begin{lemma}\label{Lem5.Q} Let the estimate \eqref{5.q-est} hold. Then, for any $\kappa>0$, the $Q_N$-component of
the solution $u(t)$ of problem \eqref{1.FCH} possesses the following estimate:
\begin{equation}\label{5.q-dis}
\|Q_Nu(t)\|_{H^{2-\kappa}}\le C_1\frac{1+t^M}{t^M}e^{-\beta t}\|Q_Nu(0)\|_{H^{-\gamma}}+C_2(1+\|g\|_H),
\end{equation}
where the constants $M,\beta>0$ and $C_1,C_2$ are independent of $N$ and $u$ (but may depend on $\kappa$).
 Moreover, for the existence
of an IM,
the strong cone property \eqref{1.cone} can be verified for the trajectories $u(t)$ satisfying
\begin{equation}\label{5.QIM}
\|Q_Nu(t)\|_{H^{2-\kappa}}\le C_2, \ t\ge0
\end{equation}
only.
\end{lemma}
Indeed, estimate \eqref{5.q-dis} follows from \eqref{5.q-est} and the parabolic regularity estimates applied to the equation
$$
A^{-\gamma}\Dt Q_Nu+AQ_Nu=Q_Ng-Q_NF(u),
$$
see the proof of Proposition \ref{Prop01.well}. The second statement is also standard and follows from
 more detailed analysis of the proof of Theorem \ref{Th02.IM}, namely, from the fact that the
  cone property is actually used for the solutions $u(t)$ with the control of $Q_N$-component
  (see, e.g., \eqref{01.constr}). More details can be found in \cite{M-PS88,KZ15,Z14}.
 \par
 Note that in contrast to the $Q_N$-component of $u(t)$, the $P_N$-component is typically unbounded
  on the IM, so we cannot assume any uniform bounds for it. Instead, we will use the extra
   map $T$ and Lemma
   \ref{Lem5.T} in order to control it.
   \par
   The next theorem can be considered as the main result of this section.

   \begin{theorem}\label{Th5.main-a} Let the nonlinearity $f$ and its spatial average $a$  satisfy
    Assumptions I,II and III  and let the truncated nonlinearity $F(u)$ be constructed via \eqref{5.comp}
     (for the properly chosen function $W$ as explained above). Then, there are
      infinitely many $N$s for which equation \eqref{1.FCH} satisfies the strong cone condition and,
      therefore, also possesses a Lipschitz IM of dimension $N$.
   \end{theorem}
\begin{proof} We need to check, following Theorem \ref{Th1.main}, that the strong cone
 property \eqref{1.cone} is satisfied for all solutions $v$ of \eqref{1.lin} with
  an extra condition \eqref{5.QIM}.
\par
   We have already verified that exactly as in Theorem \ref{Th5.a0}, the
  map $F$ is
  uniformly bounded in $H$ and its Gateaux (Hadamard) derivative is also bounded. The small change here is
   the fact that now these bounds are depending on $N$ through the term $T_N(u)$, but this term
   is not dangerous for the cone property. Indeed, due to Lemma \ref{Lem5.T}, we have
   \begin{equation}
   (T'(u)v,P_Nv-Q_Nv)=(T'(u)P_Nv,P_Nv)\le0,
\end{equation}
so we need not any extra assumptions to control it. All other terms are independent of $N$.
\par
Let us analyze the impact of every term of the derivative \eqref{5.derder} to the key estimate \eqref{1.dif}
for the cone inequality. We first note that due to the fact that all involving operators
 except of $T'(u)$ are bounded by some constant $L$ which is independent of $N$ and the term $T'(u)$ is
  not dangerous, we only need to analyze the intermediate modes.
\par
The term $l_1(u)v$ has zero spatial average, so its intermediate modes are estimated based on \eqref{5.spa}
exactly as in the proof of Theorem \ref{Th1.main}. The term $l_2(u)v$ is a scalar operator and it gives the
 truncated analogue of spatial averaging for the function $F$.
 \par
 The spatial averaging of the term $l_3(t)v$ also vanishes. Indeed,
 \begin{equation}
\|\Cal I_{k,N}l_3(u)\Cal I_{k,N}v\|_{H}\le C\|\Cal I_{k,N}W(u)\|_H\|v\|_H\le
 C(\lambda_N-k)^{-s_0/2}\|v\|_H
 \end{equation}
 and the right-hand side of it can be made arbitrarily small by chosen $N$ large enough (since $s_0>0$).
 \par
 Finally, the term $l_4(u)v$ possesses the analogous estimate
 \begin{equation}
 \|\Cal I_{k,N}l_4(u)\Cal I_{k,N}v\|_H\le C\|\Cal I_{k,N}P_Nu\|_H\|v\|_H+C\|\Cal I_{k,N}Q_Nu\|_H\|v\|_H,
 \end{equation}
but in contrast to $W(u)$ the function $u$ is not uniformly bounded in the higher energy space $H^{s_0}$. So,
 we need to argue in a more accurate way. To estimate the term containing $\|\Cal I_{k,N}Q_Nu\|_H$
  is easy due to Lemma \ref{Lem5.Q}:
$$
\|\Cal I_{k,N}Q_Nu\|_H\le C\lambda_N^{\frac{\kappa-2}2}\|Q_Nu\|_{H^{2-\kappa}}\le C_1
\lambda_N^{\frac{\kappa-2}2}.
$$
Thus, it only remains to estimate the $P_N$-component. We consider two cases: the first case
is when the estimate $\|P_Nu\|_{H^2}\le R_1$ holds. In this case everything is also easy:
$$
\|\Cal I_{k,N}P_Nu\|_{H}\le R_1(\lambda_N-k)^{-1}.
$$
The alternative case $\|P_Nu\|_{H^2}\ge R_1$ is slightly more delicate and exactly for estimating it
 we have introduced the auxiliary operator $T$. Indeed, due to this operator we now have for
  free the extra good term $-\|P_Nv\|^2_{H^1}$ which is crucial for our estimate. Namely,
  using the fact that $\theta(z)$ vanishes if
   $z$ is large enough, we get
\begin{multline}
|(P_N\Cal I_{k,N}l_4(u)\Cal I_{k,N}v,P_N\Cal I_{k,N} v)|\le
C\theta(\|u\|^2_H)\|u\|_H\|v\|_H\|P_N\Cal I_{k,N}v\|_H\le\\\le
 C_1\|v\|_H\|P_Nv\|_H\le \eb\|v\|^2_H+C_\eb(\lambda_N-k)^{-1}\|P_Nv\|_{H^1}^2,
\end{multline}
  for every $\eb>0$. Since we have the term $-\mu\|v\|^2_{H}$ in the cone inequality with $\mu$ independent of $N$, fixing
  $\eb>0$ small enough and $N$ big enough gives the desired estimate in the second case as well
  and finishes the proof of the theorem.
\end{proof}

\begin{corollary}\label{Cor5.IM} Let the assumptions of Theorem \ref{Th5.main-a} hold and let, in addition,
the solution semigroup $\bar S(t)$ associated with equation \eqref{5.maineq} possess an absorbing set $\Cal B$ which is
a bounded set of $H^s$ with $s>2$. Assume also that the constants $C_*$, $R$ and $\bar R$ are fixed in such a way that
$$
\|u\|_{H^s}\le C_*,\  \|u\|_{H^2}\le R,\ \|u\|_H\le \bar R,\ \ \forall u\in\Cal B.
$$
Then, there exist  infinitely many $N$s, such that  equation \eqref{5.maineq} possesses an IM
 of dimension $N$ in the sense of Definition \ref{Def5.IMg}. The truncated semigroup $S(t)$ is defined
  as a solution semigroup  of equation \eqref{1.FCH} with the nonlinearity $F$ defined via \eqref{5.comp}.
\end{corollary}
Indeed, this is an immediate corollary of Theorem \ref{Th5.main-a} and the fact that, by the construction,
 $F(u)=f(u)$ for all $u\in\Cal B$.
\par
To conclude this section, we discuss the $C^{1+\eb}$-smoothness of the obtained IMs.

\begin{lemma}\label{Lem5.smooth} Let the nonlinearity $f$ satisfy Assumptions I-II and let $F$ be
 constructed via \eqref{5.comp}. Then,
 \begin{equation}\label{5.F-dif}
\|F(u_1)-F(u_2)-F'(u_1)(u_1-u_2)\|_{H}\le C\|u_1-u_2\|^\eb_{H^s}\|u_1-u_2\|_H, \ u_1,u_2\in H^s,
 \end{equation}
 where the constant $C$ may depend on $N$ but is independent of $u_1,u_2$. Moreover, if $\kappa>0$
  is such that $s_0\le2-\kappa$ and $u\in H^{s}$, $s\ge2-\kappa$ satisfies
\begin{equation}\label{5.forget}
\|Q_N u\|_{H^{2-\kappa}}\le \bar C
\end{equation}
for some positive constant $\bar C$, then
 the following estimate holds:
 \begin{equation}\label{5.ex-sm}
\|F'(u)v\|_{H^{s_0}}\le C\|v\|_{H^{s_0}},
 \end{equation}
 where the constant $C$ is independent of $u\in H^{s}$, but depends on $\bar C$.
\end{lemma}
\begin{proof} Let us first check estimate \eqref{5.F-dif}. Analogously to \eqref{5.int-inc},
it is sufficient to verify that $F'\in C^\eb(H^s,\Cal L(H,H))$. Let us verify this property for every term in
\eqref{5.derder} separately. For the first term, due to Assumption II and Proposition \ref{Prop5.W}, we have
\begin{multline}
\|f'(W(u_1))W'(u_1)v-f'(W(u_2))W'(u_2)v\|_H\le\\\le \|f'(W(u_1))-f'(W(u_2)\|_{\Cal L(H,H)}
\|W'(u_1)\|_{\Cal L(H,H)}\|v\|_H+\\+\|f'(W(u_2)\|_{\Cal L(H,H)}\|W'(u_1)-W'(u_2)\|_{\Cal L(H,H)}\|v\|_H\le\\\le
 C\|W(u_1)-W(u_2)\|_{H^{s_0}}^\eb\|v\|_H+C\|W'(u_1)-W'(u_2)\|_{\Cal L(H,H)}\|v\|_H.
\end{multline}
Using the Proposition \ref{Prop5.W} again, we infer
\begin{multline}\label{5.trick}
\|W'(u_1)-W'(u_2)\|_{\Cal L(H,H)}= \|W'(u_1)-W'(u_2)\|_{\Cal L(H,H)}^{\eb}
\|W'(u_1)-W'(u_2)\|_{\Cal L(H,H)}^{1-\eb}\le\\\le C\|u_1-u_2\|_{H^s}^\eb(\|W'(u_1)\|_{\Cal L(H,H)}+
\|W'(u_2)\|_{\Cal L(H,H)})^{1-\eb}\le C\|u_1-u_2\|^\eb_{H^s}.
\end{multline}
Thus, since $s_0<s$, we have checked that
$$
\|f'(W(u_1))W'(u_1)-f'(W(u_2))W'(u_2)\|_{\Cal L(H,H)}\le C\|u_1-u_2\|_{H^{s_0}}^\eb+C\|u_1-u_2\|_{H^s}^\eb
\le C\|u_1-u_2\|_{H^s}^\eb,
$$
where the constant $C$ is independent of $u_1,u_2\in H^s$.
\par
The H\"older continuity of the terms $a(W(u))W'(u)$ and $\theta(\|u\|^2_H)a(W(u))$ can be
established analogously using Assumption II, and the term $T'_N(u)$ is also straightforward
 since it is finite-dimensional. So, it only remains to estimate $l_4(u)$. To estimate these terms,
 we actually only need to verify the uniform H\"older continuity of maps $\Psi_1:u\to\theta(\|u\|^2_H)u$ and
 $\Psi_2:u\to 2\theta'(\|u\|_{H}^2)u(u,\cdot)$   as maps from $H$ to $H$ and $\Cal L(H,H)$ and
 respectively. Let us start with the first map.
\par
Since the function $\theta$ is smooth and has a finite support, the map $\Psi_1(u)$ is at least Gateaux
differentiable and its derivatives is given by
$$
\Psi'_1(u)v:=\theta(\|u\|^2_H)v+2\theta'(\|u\|_H^2)(u,v)u
$$
and, therefore,
$$
\|\Psi'_1(u)\|_{\Cal L(H,H)}\le \max_{z\in\R_+}\{|\theta(z)|\}+2\max_{z\in\R_+}\{z|\theta'(z)|\}\le C.
$$
Since $\|\Psi_1(u)\|_{H}\le C$, we end up with
$$
\|\Psi_1(u_1)-\Psi_1(u_2)\|_{H}\le C\|u_1-u_2\|^\eb_{H}\le C\|u_1-u_2\|_{H^s}^\eb,
$$
where the constant $C$ is independent of $u_1$ and $u_2$. Let us now look at the second map $\Psi_2$. Analogously,
its Gateaux derivative reads
$$
\Psi_2'(u)[w,v]=2\theta'(\|u\|^2_H)u(w,v)+2\theta'(\|u\|^2_H)w(u,v)+4\theta''(\|u\|^2_H)(w,u)(u,v)u,
$$
and using the fact that $\theta$ has a finite support, we get
$$
\|\Psi_2'(u)[w,v]\|_H\le C\|w\|_H\|v\|_H,
$$
where $C$ is independent of $u,v,w\in H$. Since $\Psi_2$ is bounded as a map from $H$ to $\Cal L(H,H)$, we
 infer from here that
 $$
\|\Psi_2(u_1)-\Psi_2(u_2)\|_{\Cal L(H,H)}\le C\|u_1-u_2\|_{H^s}^\eb.
 $$
 The obtained estimates, together with Assumption II and
 Proposition \ref{Prop5.W}, give
$$
\|l_4(u_1)-l_4(u_2)\|_{\Cal L(H,H)}\le C\|u_1-u_2\|_{H^s}^\eb
$$
and finish the proof of estimate \eqref{5.F-dif}.
\par
Let us verify estimate \eqref{5.ex-sm}. This estimate is an almost immediate corollary of Assumption II c)
 and Proposition \ref{Prop5.W}. The only problem which appears here is related with the term $l_4(u)$.
 Indeed, arguing as before, we get
$$
\|l_4(u)\|_{\Cal L(H^{s_0},H^{s_0})}\le C\(\theta(\|u\|^2_H)\|u\|_H+C|\theta'(\|u\|^2_H)|\)\|u\|_{H^{s_0}},
$$
where $C$ is independent of $u$.   To handle this term, we write
 $$
 \|u\|_{H^{s_0}}\le \|P_Nu\|_{H^{s_0}}+\|Q_Nu\|_{H^{s_0}}\le C_N\|u\|_{H}+\|Q_Nu\|_{H^{s_0}}.
 $$
 Therefore, since $\theta$ has a finite support, we end up with
 $$
 \|l_4(u)\|_{\Cal L(H^{s_0},H^{s_0})}\le C_N(1+\|Q_Nu\|_{H^{s_0}}).
 $$
Since $s_0\le 2-\kappa$, the right-hand side is bounded due to condition \eqref{5.forget}. This
finishes the proof of estimate \eqref{5.ex-sm} as well as the lemma.
\end{proof}

\begin{corollary}\label{Cor5.IMsm} Let the assumptions of Theorem \ref{Th5.main-a} and Lemma \ref{Lem5.smooth} hold and
let, in addition $s_0<2$ and $s_0<s<s_0+2$. Then  there exists an infinite sequence of $N$s
such that problem  \eqref{5.maineq} possesses an $N$-dimensional IM which is $C^{1+\eb_N}$-smooth
for some $\eb_N>0$.
\end{corollary}
Indeed, this result follows from Lemma \ref{Lem5.smooth}, Theorem \ref{Th02.IMsm} and Remark \ref{Rem02.sm}.

\section{Spatial averaging in the case of periodic boundary conditions}\label{s5}

In this section, we discuss the spatial averaging Assumption III in the most usual
(from the point of view of applications) case where $A$ is the Laplacian $A=-\Dx$ in the 3D
 domain $\Omega=(-\pi,\pi)^3$ endowed with periodic boundary conditions. In this case the eigenvalues
  $\lambda_n$ of $A$ are naturally parameterised by triples $\vec n:=(q,l,m)$ of integer numbers:
  \begin{equation}
\lambda_{\vec n}=q^2+l^2+m^2,\ \ \ e_{\vec n}:=e^{ix.\vec n}=e^{ix_1q+ix_2l+ix_3m}.
  \end{equation}
Then, the Fourier series \eqref{01.F} become the classical Fourier expansions. We will use the notation
$\{\lambda_{\vec n}\}_{\vec n\in\Bbb Z^3}$ for this parametrization keeping the notation
 $\{\lambda_n\}_{n\in\Bbb N}$ for the  parametrization in the non-decreasing order used in previous sections.
\par
 Note also that all
 $\lambda_n$ are integer, so the distance from non-identical eigenvalues is at least one and due to
 the Gauss theorem
 about sums of squares there are no spectral gaps of size more than 3, see \cite{M-PS88,Z14}
  for more details. Thus,
 in general, the spectral gap conditions are {\it not satisfied} in all of examples considered below.
 \par
 We also recall that due to the Weyl asymptotic $\lambda_n\sim Cn^{2/3}$, so
 the key condition \eqref{5.sm-eat}
 is satisfied if and only if
\begin{equation}
 s>s_0+\frac32.
\end{equation}
There is also a small problem here related with possible zero eigenvalue which corresponds
 to $\vec n=(0,0,0)$. This can be overcome in two alternative ways. First, we may consider $A=1-\Dx$ instead
  of $A=-\Dx$ which removes the problem up to the nonessential shift of the spectrum. This is typically done, say,
  for reaction-diffusion equations. Alternatively, we may work in spaces with zero mean which is natural for
  Navier-Stokes or Cahn-Hilliard type problems. In this case the problem does not arise at all.
\par
The spatial averaging method takes its origin in the following number theoretic result which claims that the sums of
 3 squares of integers are distributed irregularly enough.

\begin{proposition}\label{Lem6.sp} Let
\begin{equation}
\Cal C^k_N:=\{\vec l\in\Bbb Z^3\,:\, N-k\le|\vec l|^2\le N+k\},\ \ \Cal B_r:=\{\vec l\in\Bbb Z^3\,:\,
|\vec l|\le r\}.
\end{equation}
Then, for every $k>0$ and $r>0$, there exist infinitely many $N\in\Bbb N$ such that
\begin{equation}\label{6.spav}
\(\Cal C^k_{N}-\Cal C^k_N\)\cap \Cal B_r=\{0\}.
\end{equation}
\end{proposition}
The proof of this proposition is given in \cite{M-PS88}.
\par
With a slight abuse of notations, we redefine the projector $\Cal I_{k,N}$ as follows:
\begin{equation}
\Cal I_{k,N}v:=\sum_{\vec n\in \Cal C^k_N}(v,e_{\vec n})e_{\vec n}.
\end{equation}
Being pedantic, we should write $\Cal I_{k,N'}$ in the left-hand side of this formula, where $N'=N'(N)$
is defined by
$$
N':=\max\{M\in\Bbb N, \lambda_M\le N\}.
$$
However, the difference between $N$ and $N'$ is not essential for us and to simplify the notations, we
 ignore it. The next proposition is also crucial for the spatial averaging machinery.
\par
\begin{proposition}\label{Prop6.mult} Let $\Cal N_{\psi}:L^2(\Omega)\to L^2(\Omega)$ be the operator of
point-wise multiplication on a function $\psi\in H^{s_0}$ for some $s_0>\frac32$:
\begin{equation}\label{6.mult}
(\Cal N_\psi v)(x):=\psi(x)v(x).
\end{equation}
Then, the operator $\Cal N_{\psi}$ satisfies the spatial averaging property in the following sense: for every $k>0$ and
every $\delta>0$ there exists an infinitely many $N\in\Bbb N$ such that
\begin{equation}\label{6.sp}
\|\Cal I_{k,N}\Cal N_\psi\Cal I_{k,N}v-a\Cal I_{k,N}v\|_{L^2}\le\delta \|v\|_{L^2}, \ v\in L^2,
\end{equation}
where $a=\<\psi\>:=\frac1{2\pi^3}\int_{(-\pi,\pi)^3}\psi(x)\,dx$.
\end{proposition}
\begin{proof} Although this result is well-known, its proof is crucial for understanding the
spatial averaging technique, so we give some details below following mainly \cite{Z14}.
\par
The multiplication $\psi(x)v(x)$ is a convolution in Fourier modes, so the corresponding
 Fourier coefficients $[\psi v]_{\vec m}$, $\vec m\in\Bbb Z^3$ satisfy
\begin{equation}
[\psi v]_{\vec m}=\sum_{\vec l\in\Bbb Z^3} \psi_{\vec m-\vec l}\, v_{\vec l}
\end{equation}
and, due to condition \eqref{6.spav},
\begin{equation}
\Cal I_{k,N}\((\psi-\<\psi\>) \Cal I_{k,N}v\)=\Cal I_{k,N}\(\psi_{>r}\Cal I_{k,N}v\),
\end{equation}
where $\psi_{>r}(x):=\sum_{\vec l\notin \Cal B_r}\psi_{\vec l} e^{i\vec l.x}$. Therefore,
\begin{equation}
\|\Cal I_{k,N}\((\psi-\<\psi\>) \Cal I_{k,N}v\)\|_H\le
\|\(\psi_{>r}\Cal I_{k,N}v\)\|_H\le \|\psi_{>r}\|_{L^\infty}\|v\|_H.
\end{equation}
Thus, we only need to check that
\begin{equation}
\lim_{r\to\infty}\|\psi_{>r}\|_{L^\infty}=0.
\end{equation}
To verify this property, we use the interpolation inequality
$$
\|\psi_{>r}\|_{L^\infty}\le C\|\psi_{>r}\|_{L^2}^{\kappa}\|\psi_{>r}\|_{H^{s_0}}^{1-\kappa}
$$
for the properly chosen $\kappa=\kappa(\alpha)\in(0,1)$ (here we have used that $s_0>\frac32$),
together with the standard inequality $\|\psi_{>r}\|_{H^{s_0}}\le C\|\psi\|_{H^{s_0}}$. Thus, we have
$$
\|\psi_{>r}\|_{L^\infty}\le C\|\psi_{>r}\|_{L^2}^\kappa\|\psi\|_{H^{s_0}}^{1-\kappa}\le
C r^{-\kappa s_0}\|\psi_{>r}\|_{H^{s_0}}^\kappa\|\psi\|_{H^{s_0}}^{1-\kappa}\le
 Cr^{-\kappa s_0}\|\psi\|_{H^{s_0}}
$$
and the proposition is proved.
\end{proof}
At the next step, we consider the particular case where the function $\psi$ has zero mean. Then the
 class of operators with spatial averaging property can be essentially extended.

\begin{proposition}\label{Prop6.0-mean} Let $A_1,A_2$ be two linear operators which commute with the
 operator $A:=-\Dx$ with periodic boundary conditions (more precisely, we need
 the commutation of them with spectral projectors $\Cal I_{k,N}$)
  and let
 \begin{equation}
A_1\in\Cal L(H^{-\beta},H),\ \ A_2\in \Cal L(H,H^{-\beta}),\
 \end{equation}
 for some $\beta\in[0,1]$. Assume also that $\psi\in H^{s_0}$ for some $s_0>\frac32$ and $\<\psi\>=0$.
 Then the operators
 $$
\Cal N_{A_1,\psi,A_2}:=A_1\circ\Cal N_\psi\circ A_2 \ \ \text{and}\ \ \Cal N_{A_1,A_2\psi}:=A_1\circ\Cal N_{A_2\psi}
 $$
 satisfy the spatial averaging property \eqref{6.sp} with $a=0$.
\end{proposition}
\begin{proof} Let us start with the first operator.
Arguing as in the proof of Proposition \ref{Prop6.mult} and using that $\Cal I_{k,N}$ commute with $A_1$ and
 $A_2$, we see that it is sufficient to show that
$$
\lim_{r\to\infty}\|\Cal N_{A_1,\psi_{>r},A_2}\|_{\Cal L(H,H)}=0.
$$
In turn, to this end, we only need to show that
$$
\lim_{r\to\infty}\|\Cal N_{\psi_{>r}}\|_{\Cal L(H^{-\beta},H^{-\beta})}=0.
$$
To check this property, we will use the following version of the Kato-Ponce inequality,
see \cite{BCD11,KP88}:
$$
\|\psi w\|_{H^{\beta}}\le C\|\psi\|_{L^\infty}\|w\|_{H^{\beta}}+C\|\psi\|_{H^{\beta,q}}\|w\|_{L^p},
$$
where $\frac12=\frac1q+\frac1p$. We fix the exponents
$\frac1p=\frac12-\frac\beta3$, $\frac1q=\frac\beta3$ in
 order to have the Sobolev embeddings $H^\beta\subset L^p$
 and $H^{s'}\subset H^{\beta,q}$ for all $s'>\frac32$. This gives us the estimate
 \begin{equation}
\|\psi w\|_{H^\beta}\le C\|\psi\|_{H^{s'}}\|w\|_{H^\beta}
 \end{equation}
 and, therefore, taking $s'\in(3/2,s_0)$ and using the standard trick with adjoint operator, we have
$$
\|\Cal N_{\psi_{>r}}\|_{\Cal L(H^{-\beta},H^{-\beta})}\le C\|\psi_{>r}\|_{H^{s'}}\le
 C r^{s'-s_0}\|\psi\|_{H^{s_0}}
$$
which finishes the proof of the proposition for the operator $\Cal N_{A_1,\psi,A_2}$.
\par
Let us now study the second operator $\Cal N_{A_1,A_2\psi}$. Using again that $A_1$, $A_2$
commute with $\Cal I_{k,N}$ and arguing as in the proof of Proposition \ref{Prop6.mult}, we see that,
we only need to verify that
$$
\lim_{r\to\infty}\|\Cal N_{A_2(\psi_{>r})}\|_{\Cal L(H,H^{-\beta})}=0.
$$
To verify this, we use the Sobolev embedding $L^q\subset H^{-\beta}$ for $\frac1q=\frac12+\frac\beta3$,
$H^\mu\subset L^p$ for $\frac1p=\frac12-\frac\mu3$, where $\frac1q=\frac12+\frac1p$,
i.e., $\mu=\frac32-\beta$, together with H\"older's inequality. This gives
\begin{multline}
\|A_2(\psi_{>r})v\|_{H^{-\beta}}\le C\|A_2(\psi_{>r})v\|_{L^q}\le
C\|A_2\psi_{>r}\|_{L^p}\|v\|_H\le C\|A_2\psi_{>r}\|_{H^{3/2-\beta}}\|v\|_{L^2}\le\\\le
C\|\psi_{>r}\|_{H^{3/2}}\|v\|_H\le Cr^{3/2-s_0}\|\psi\|_{H^{s_0}}\|v\|_H
\end{multline}
and the proposition is proved.
\end{proof}
We conclude this section by verifying the spatial averaging property as well as other properties stated in
Assumptions I-II for a number of concrete nonlinearities related with our applications to  reaction-diffusion,
 Cahn-Hilliard and Navier-Stokes equations.

\begin{example}\label{Ex6.RDE} Let us consider the local scalar nonlinearity $f(u)$ for some $f\in C^4(\R,\R)$.
This will correspond to the case of
reaction-diffusion equation \eqref{5.maineq} with $\gamma=0$. In this case the derivative $f'(u)v$ is
 a multiplication operator on a function $\psi=f'(u)$. Thus, according to Proposition \ref{Prop6.mult},
 the spatial averaging assumption (Assumption III) will be satisfied with $a:=\<f'(u)\>$ if we take
 $\frac32<s_0<2$. Using the fact that $H^{s_0}$ is an algebra if $s_0>\frac32$, we see that the
 other regularity assumptions (Assumptions I-II) are also automatically satisfied if
 \begin{equation}\label{6.smoothness}
 \frac32<s_0<2,\ \ s_0+\frac32<s<s_0+2<4.
\end{equation}
 In this case, we take $A=1-\Dx$ in order to remove zero eigenvalue.
 \par
 Thus, to verify the existence of an IM of smoothness $C^{1+\eb}$ for this type of the nonlinearity
  (according to Corollary \ref{Cor5.IMsm}) it is enough to verify that the corresponding
   equation \eqref{5.maineq}
  possesses an absorbing ball in the space $H^s$ where $s$ satisfies
   \eqref{6.smoothness}. This will be discussed
  in the next section.
\end{example}
\begin{example}\label{Ex6.CH} Let us modify slightly the previous example to adapt it to
 the case of Cahn-Hilliard
type equations. Namely, we will consider the space $H=\{u\in L^2,\<u\>=0\}$ and consider the nonlinearity
\begin{equation}
f(u)-\<f(u)\>.
\end{equation}
The extra non-local term $\<f(u)\>$ has 1 dimensional range and does not affect the spectral
averaging property as well
 as other regularity properties of the nonlinearity. Thus, to get the existence
 of $C^{1+\eb}$-smooth IMs, we only need to get
  the absorbing set in $H^s$ satisfying \eqref{6.smoothness}.
\end{example}

\begin{example}\label{Ex6.NS} We now consider the Navier-Stokes type nonlinearities. We assume that
\begin{equation}\label{6.divfree}
H:=\{u\in [L^2(\Omega)]^3,\ \<u\>=0,\ \divv u=0\}
\end{equation}
and denote by $P$ the Leray orthoprojector from $[L^2(\Omega)]^3$ to $H$. It is well-known that in the
case of periodic boundary conditions, the Leray projector $P$ commutes with the Laplacian and, therefore,
the Stokes operator $A=-P\Dx$ is just a restriction of the Laplacian to the space $H$ of divergence-free
 vector fields. Thus, the results of this section on spatial averaging and, in particular, Proposition
  \ref{Prop6.0-mean} remain valid for the Stokes operator as well.
  \par
  Let us now consider a special class of modified Navier-Stokes nonlinearities. First, in the spirit of
  Leray-$\alpha$ model, we define
  $$
  \bar u:=(1-\alpha\Dx)^{-\bar\gamma}u
  $$
  for some $\bar\gamma\ge0$ and then we consider the nonlinearity
  \begin{equation}\label{6.NSnon}
f(u):=P(-\Dx)^{-\gamma}[(u\cdot\Nx)\bar u]=P(-\Dx)^{-\gamma}[(u\cdot\Nx)(1-\alpha\Dx)^{-\bar\gamma}u]
  \end{equation}
  for the corresponding $\bar\gamma,\gamma\ge0$. In order to have zero order nonlinearity, we need to add
  extra condition $\gamma+\bar\gamma\ge\frac12$. Since the situation where this inequality is strict can
  be only simpler, we will assume from now on that
\begin{equation}
\gamma+\bar\gamma=\frac12.
\end{equation}
As we will see below, the limit case $\bar\gamma=0$, $\gamma=\frac12$ corresponds to
 hyperviscous Navier-Stokes equation and another limit case $\gamma=0$, $\bar\gamma=\frac12$
 gives us the so-called Leray-$\alpha$-Bardina model.
\par
Note that the gradient  also commutes with the Leray projector and with the Laplacian, so
the derivative
\begin{multline}\label{6.N-der}
f'(u)v=P(-\Dx)^{-\gamma}[(u\cdot\Nx)(1-\alpha\Dx)^{-\bar\gamma}v]+\\+P(-\Dx)^{-\gamma}[(v\cdot\Nx)
(1-\alpha\Dx)^{-\bar\gamma}u]:=B(u,v)+B(v,u)
\end{multline}
can be written as a sum of operators considered in Proposition \ref{Prop6.0-mean} with $\beta=2\gamma$.
Moreover, the spatial
averaging of every such a term is equal to zero due to the assumption that $H$ consists of
 functions with zero mean. Thus, spatial averaging Assumption III is satisfied if $s_0>\frac32$.
\par
Let us verify the regularity assumptions (Assumptions I-II) for the Navier-Stokes nonlinearity $f$
given by \eqref{6.NSnon}. To this end, we recall that, we have actually proved in Proposition
 \ref{Prop6.0-mean} that
\begin{equation}\label{6.mainest}
\|f'(u)v\|_H\le C\|u\|_{H^{s_0}}\|v\|_H.
\end{equation}
 Moreover, since the map $u\to f'(u)$ is a {\it linear} continuous
  map from $H^{s_0}$ to $\Cal L(H,H)$, its H\"older continuity (as well as even
   $C^\infty$-smoothness) is also an immediate corollary of \eqref{6.mainest}. Thus, we have
    verified properties a) and b) of Assumption II. To verify Assumption I, it is enough to note that $f$ is a homogeneous quadratic form,
    so by Euler identity,
    $$
    f(u)=\frac12 f'(u)u
    $$
    and Assumption I also follows from \eqref{6.mainest}. Thus, we only need to verify property c)
    of Assumption II, namely, that $f'(u)$ is a bounded operator from $H^{s_0}$ to $H^{s_0}$. To this end,
    we will use again the Kato-Ponce formula together with the proper Sobolev embeddings. Namely,
\begin{multline}\label{6.KP1}
    \|B(u,v)\|_{H^{s_0}}\le C\|(u\cdot\Nx)(1-\alpha\Dx)^{-\bar\gamma}v\|_{H^{s_0-2\gamma}}\le\\\le
     C\|u\|_{L^\infty}\|v\|_{H^{s_0-2\gamma-2\bar\gamma+1}}+
     \|u\|_{H^{s_0-2\gamma,p}}\|v\|_{H^{1-2\bar\gamma,q}},
\end{multline}
where $\frac1p+\frac1q=\frac12$. The first term in the right-hand side of \eqref{6.KP1} is under
 control due to embedding $H^{s_0}\subset L^\infty$ for $s_0>\frac32$. To estimate the second
  one, we fix $\frac1p=\frac12-\frac{2\gamma}3$. Then, $H^{s_0}\subset H^{s_0-2\gamma,p}$ and
  $\frac1q=\frac{2\gamma}3$. Therefore, $s_0>\frac32$ implies that
  $$
  \frac1q>\frac12-\frac{s_0-(1-2\bar\gamma)}3
  $$
  and $H^{s_0}\subset H^{1-2\bar\gamma,q}$. Finally, \eqref{6.KP1} implies that
  $$
  \|B(u,v)\|_{H^{s_0}}\le C\|u\|_{H^{s_0}}\|v\|_{H^{s_0}}
  $$
  and property c) of Assumption II is also verified.
\par
 Thus, for the existence of $C^{1+\eb}$-smooth IM for such
    nonlinearities it is sufficient to verify the existence of an absorbing ball in the space $H^s$.
     This will be discussed in the
     next section.
\end{example}

\section{Applications}\label{s6}
In this section we consider the applications of our general theory to several classes of equations. Note that
 the regularity and spatial averaging assumptions for the nonlinearities considered are already verified in the previous section,
 so it only remains to check the dissipativity in the proper spaces.

 \subsection{Scalar Reaction-Diffusion equation}\label{ss7.1} Let us consider the equation
 \begin{equation}\label{7.RDS}
\Dt u=\Dx u-u-f(u)+g,\ \ u\big|_{t=0}=u_0, \ g\in L^2((-\pi,\pi)^3)
 \end{equation}
endowed with the periodic boundary conditions. This equation is of the form \eqref{5.maineq} with
 $\gamma=0$ and $Au:=-\Dx u+u$. Let us pose the following conditions on the scalar function $f$:
 \begin{equation}\label{7.f}
\begin{cases}
1.\ f\in C^4(\R,\R),\\
2.\ f(u)u\ge -C,\ \ u\in\R,\\
3.\ f'(u)\ge-K,\ \ u\in\R.
\end{cases}
 \end{equation}
Then, as well-known (see, e.g., \cite{BV92,S10,T97}),  problem \eqref{7.RDS} is well-posed in the phase
 space $H=L^2((-\pi,\pi)^2)$ and generates a dissipative semigroup $\bar S(t):H\to H$:
 \begin{equation}\label{7.dis}
\|u(t)\|_{H}\le C\|u(0)\|_{H}e^{-\kappa t}+C(\|g\|_{H}+1).
 \end{equation}
 Moreover, the following smoothing property holds:
 \begin{equation}
\|u(t)\|_{H^2}\le C\frac{t+1}t \(e^{-\kappa t}\|u_0\|_{H}+\|g\|_H+1\)
 \end{equation}
 and therefore the invariant bounded absorbing set $\Cal B\subset H^2$ exists.
\par
However, we need a bit more regularity, namely, the absorbing ball in the space $H^s$ with $3<s<4$.
 To get this, we either need to require $g\in H^2$ which looks a bit restrictive or to use the
  standard  trick with introducing the auxiliary function $G\in H^2$ as a solution of the
  following elliptic problem
  \begin{equation}
\Dx G-G+g=0.
  \end{equation}
  Obviously, the solution of this problem exists and introducing the new variable $v=u-G$, we get
   the equivalent equation for $v$:
  \begin{equation}
\Dt v=\Dx v-v-f(v+G).
  \end{equation}
  The advantage of this equation is that, due to the fact that $H^2$ is algebra, we have the control of the
  $H^2$-norm of $f(v+G)$ on the $H^2$-absorbing ball and then from linear parabolic smoothing property,
  we get the desired absorbing ball for $v$ in the space $H^{4-\eb}$ for all $\eb>0$ which is enough
   for our purposes. Since the shift $u\to u+G$ does not affect the Assumptions I-III for the
    nonlinearity $f$
   in Example \ref{Ex6.RDE}, we end up with the following result.

   \begin{corollary}\label{Co7.RDS} Let $g\in H$ and the nonlinearity $f$ satisfy assumptions
    \eqref{7.f}. Then, there are infinitely many values of $N$, such that the
   reaction-diffusion equation \eqref{7.RDS} possesses  $C^{1+\eb_N}$-smooth IM in the sense of
   Definition \ref{Def5.IMg}. The truncated nonlinearity can be chosen in the form of \eqref{5.comp}
   with exponents $s_0$ and $s$ satisfying \eqref{6.smoothness}.
   \end{corollary}
 \begin{remark} Of course, this result is well-known and has been first obtained in the
 pioneering paper \cite{M-PS88}, see also \cite{KZ15,Z14} for the smoothness of the manifold.
  In fact, $C^4$-smoothness
  here also can be relaxed, but we start from this example just in order to show that our
   general theory covers
   this classical result.
 \end{remark}

\subsection{Cahn-Hilliard type equations}\label{ss7.2} We now turn to more interesting problem related with
generalizations of the Cahn-Hilliard equation:
\begin{equation}\label{7.CH}
\Dt u+(-\Dx)^\gamma(-\Dx u+f(u)+g)=0,\ \ u\big|_{t=0}=u_0
\end{equation}
in $\Omega=(-\pi,\pi)^3$ endowed with periodic boundary conditions. The case $\gamma=1$ corresponds to
 the classical Cahn-Hilliard equation, see \cite{T97,EFZ04,MZ08} and references therein for more details. The case $0<\gamma<1$ is
  the so-called fractional Cahn-Hilliard equation which is of a big current interest, see \cite{ASS16} and
   references therein. The other choices of $\gamma>0$ are also interesting, for instance, $\gamma=2$ corresponds to the
    so-called $6$th order Cahn-Hilliard equation, see \cite{Mir14} and references therein.
\par
This equation has a natural (mass) conservation law:
$$
\frac d{dt}\<u(t)\>=0,
$$
so, without loss of generality, we may assume that $\<u\>=0$ and consider $A=-\Dx$ in the space
$H=\{u\in L^2((-\pi,\pi)^3),\ \<u\>=0\}$. Then zero eigenvalue disappears and the operator $A$ becomes positive
 definite. We also assume that $g\in H$. Then equation \eqref{7.CH} is equivalent to the following one:
 \begin{equation}\label{7.FCH}
\Dt (-\Dx)^{-\gamma}u-\Dx u+f(u)-\<f(u)\>+g=0,
 \end{equation}
 so the equation is indeed in the form of \eqref{5.maineq}. We pose exactly the same conditions \eqref{7.f}
  to the nonlinearity $f$. Then, verification of Assumptions I-III is also exactly the same as
  for the case of the reaction-diffusion equation (with the same values of $s_0$ and $s$) since the presence of the extra one-dimensional
  term $\<f(u)\>$ changes nothing. Thus, we only need to check the existence of the absorbing ball in $H^s$.
  Moreover, we only need this absorbing ball in $H^2$ since the further regularity can be obtained in
  a straightforward way using the linear parabolic regularity estimates
  (similarly to the case of a reaction-diffusion equation). So, we will briefly discuss below only the
  well-posedness and $H^2$-regularity and dissipativity of solutions of \eqref{7.FCH} in $H^2$. This is a
   straightforward generalization of the standard Cahn-Hilliard theory (for $\gamma=1$),
   see \cite{T97,EFZ04,MZ08} for more details.
\par
Similarly to the classical CH-equation, the natural phase space for problem \eqref{7.FCH}
is $H^{-\gamma}$ since exactly in this case we may utilize the monotonicity of the nonlinearity $f$ and get
 nice estimates in the same way as in Proposition \ref{Prop01.well}. Indeed, multiplying
 equation \eqref{7.FCH} on $u$, integrating over $x$ and using that $f(u).u\ge-C$, we get the
  following analogue of dissipative estimate \eqref{01.dis}:
  \begin{equation}\label{7.fu}
  \|u(t)\|^2_{H^{-\gamma}}+\int_t^{t+1}\|u(s)\|^2_{H^1}+|f(u(s)).u(s)|\,ds\le
   Ce^{-\kappa t}\|u_0\|^2_{H^{-\gamma}}+C(1+\|g\|_H^2)
  \end{equation}
for some positive constants $C$ and $\kappa$.
\par
   Moreover, writing the equation for differences $v(t)=u_1(t)-u_2(t)$
  of two solutions, multiplying it by $v(t)$, using that $f'(u)\ge-K$, and arguing exactly as in the proof
  of Proposition \ref{Prop01.well}, we get the global Lipschitz continuity \eqref{01.lip}. The existence
   of a solution can be obtained in a standard way using, e.g., the Galerkin approximations,
   see \cite{BV92,T97}.
  Thus, we have verified the global well-posedness and dissipativity of the solution
   semigroup $\bar S(t):H^{-\gamma}\to H^{-\gamma}$ associated with the equation \eqref{7.CH}.
\par
Let us now discuss the smoothing property. First, multiplying equation \eqref{7.FCH} by $t\Dt u$, we get
\begin{equation}\label{7.grad}
\frac d{dt}\(\frac12 t\|u\|^2_{H^1}+t(\Phi(u),1)+t(g,u)\)+t\|\Dt u(t)\|^2_{H^{-\gamma}}=
\frac12 \|u\|^2_{H^1}+(\Phi(u),1)+(g,u),
\end{equation}
where $\Phi(z):=\int_0^zf(s)\,ds$. Using the elementary inequality
$$
\Phi(u)\le f(u).u+\frac K2|u|^2
$$
and \eqref{7.fu} for estimating the terms in the right-hand side of \eqref{7.grad}, we end up with
\begin{equation}\label{7.h1}
\|u(t)\|^2_{H^1}+\int_t^{t+1}\|\Dt u(s)\|^2_{H^{-\gamma}}\,ds\le C\frac{t+1}{t}\(e^{-\kappa t}\|u_0\|^2_{H^{-\gamma}}+1+\|g\|^2_H\)
\end{equation}
for some positive $C$ and $\kappa$. This estimate gives us the absorbing ball for the semigroup
 $\bar S(t)$ in $H^1$, but to get the desired $H^2$-smoothing, we need more steps.

   At the next step, we differentiate equation \eqref{7.FCH}
by $t$ and denote $\theta(t):=\Dt u(t)$. Then multiplying the result by $t^2\theta(t)$,
using the fact that $f'(u)\ge-K$ and the estimate for the integral norm of $\theta(t)$ obtained in \eqref{7.h1}
analogously to \eqref{01.in-sm}, we get
\begin{equation}\label{7.dt}
\|\theta(t)\|_{H^{-\gamma}}^2\le C \frac{t^2+1}{t^2}\(e^{-\kappa t}\|u(0)\|^2_{H^{-\gamma}}+1+\|g\|_H^2\)
\end{equation}
for some positive $C$ and $\kappa$.
\par
Finally, analogously to \eqref{01.el-l}, we write our problem as an elliptic problem
$$
\Dx u(t)-f(u(t))+\<f(u(t))\>=A^{-\gamma}\theta(t)+g
$$
and multiply this equation by $\Dx u(t)$ followed by integration over $x$. Then, using the
obtained estimates \eqref{7.h1} and \eqref{7.dt} for $\theta(t)$ and the assumption $f'(u)\ge-K$, we arrive at
\begin{equation}
\|u(t)\|_{H^2}^2\le C \frac{t^2+1}{t^2}\(e^{-\kappa t}\|u(0)\|^2_{H^{-\gamma}}+1+\|g\|_H^2\)
\end{equation}
which gives us the desired existence of an absorbing ball in the space $H^2$. Since $H^2((-\pi,\pi)^3)$
is an algebra the further smoothing estimates are straightforward and we have proved the following result.

   \begin{corollary}\label{Co7.FCH} Let $g\in H$ and the nonlinearity $f$ satisfy assumptions
    \eqref{7.f}. Then, there are infinitely many values of $N$, such that the
   Cahn-Hilliard type  equation \eqref{7.CH} possesses  $C^{1+\eb_N}$-smooth IM in the sense of
   Definition \ref{Def5.IMg}. The truncated nonlinearity can be chosen in the form of \eqref{5.comp}
   with an extra term $\<f(W(u))\>$
   and  exponents $s_0$ and $s$ satisfying \eqref{6.smoothness}.
   \end{corollary}
\begin{remark}\label{Rem7.CH} For the case $\gamma=1$ which corresponds to the classical
Cahn-Hilliard equation, this result has
 been established in \cite{KZ15}. However, for other values of $\gamma>0$ this result
 seems new. One of the main achievements of our approach is that we can treat all the
  cases $\gamma>0$ as well as $\gamma=0$
  from the unified point of view.
\end{remark}

\subsection{Modified Navier-Stokes equations}\label{ss7.3} We now turn to the other class of examples related
with Navier-Stokes equations which also fits our general theory. Namely, we will consider
the following class of modified 3D Navier-Stokes equations:
\begin{equation}\label{7.NS}
\begin{cases}
\Dt u+(u,\Nx\bar u)+(-\Dx)^{1+\gamma} u+\Nx p=g, \ \ u\big|_{t=0}=u_0,\\
\divv u=0,\ \ \bar u=(1-\alpha\Dx)^{-\bar\gamma}u.
\end{cases}
\end{equation}
The case $\gamma=\bar\gamma=0$ corresponds to the classical 3D Navier-Stokes problem. However, the
global well-posedness
 of this problem is out of reach of the modern theory and is actually one of the Millennium
 problems, see \cite{Fef06}
 and references therein, so some modifications/regularisations look unavoidable. Introducing the
  truncated variable $\bar u$ is in the spirit of Leray $\alpha$-regularization or
  the so-called Bardina model,
  see \cite{A13,BFR80,CHOT05} and references therein. The term $(-\Dx)^{1+\gamma}u$ with $\gamma>0$ gives an
  alternative popular type of regularization - the so-called hyperviscous regularization of the
   Navier-Stokes problem, see \cite{HLT10,LL06,OT07}.
   \par
   The IM theory requires extra assumptions on the exponents $\gamma$ and $\bar\gamma$ in
   comparison with well-posedness. For instance, for the 2D case the classical Navier-Stokes equations are
   globally well-posed, but for the existence of an IM for periodic boundary conditions, we still need
   $\bar \gamma\ge\frac12$, see \cite{HGT15}, and the existence of an IM for $\bar\gamma<\frac12$ is
    still an open problem.
\par
For problem \eqref{7.NS} the borderline for the IM theory is given by the condition
\begin{equation}\label{7.Img}
\gamma+\bar\gamma=\frac12,\ \gamma\in[0,\frac12].
\end{equation}
As we will see, under this assumption, equation \eqref{7.NS} can be reduced to our abstract
equation \eqref{5.maineq} and the existence of an IM follows from the general theory. By this reason,
in order
 to avoid technicalities, we restrict ourselves to consider the case of equality
  \eqref{7.Img} only. The case
  when $\gamma+\bar\gamma>\frac12$ is similar (but simpler since we have the extra regularity
   for the nonlinearity) and also fits our theory, but the case $\gamma+\bar\gamma <\frac12$ is
   out of reach of the
    theory and remains an open problem.
\par
 To embed this problem into a general theory developed above, we take as in Example \ref{Ex6.NS} the
 space $H$ as the space of divergent free vector fields defined by \eqref{6.divfree} and rewrite
 \eqref{7.NS} in the equivalent form
 \begin{equation}\label{7.NSmod}
A^{-\gamma}\Dt u+Au+A^{-\gamma}P[(u,\Nx)(1-\alpha\Dx)^{-\bar\gamma}u]=A^{-\gamma}g,\ \ u\big|_{t=0}=u_0,
 \end{equation}
 where $P$ is a Leray projector to the divergent free vector fields and $A$ is a Stokes
 operator which coincides in the
 case of periodic boundary conditions with the restriction of the minus Laplacian to
 the space $H$. We also assume that $g\in H$.
 \par
Equation \eqref{7.NSmod} has the form of \eqref{5.maineq} with the nonlinearity \eqref{6.NSnon}
considered in Example \ref{Ex6.NS}. As we have established there, this nonlinearity satisfies our
Assumptions I-III for the IM-existence theorem with exponents $s$ and $s_0$ satisfying
 \eqref{6.smoothness}. Thus, in order to get the desired existence of $C^{1+\eb}$ IM for
 problem \eqref{7.NS}, we only need to verify the well-posedness in the proper phase space
  (which in general need not
  to coincide with $H^{-\gamma}$) and the existence of an absorbing set, bounded in $H^s$ for some $3<s<4$.
\par
The well-posedness and regularity theory for the Navier-Stokes type equations of the form \eqref{7.NS}
is also well-understood nowadays, so we will restrict
 ourselves only to a brief exposition indicating the key features,
 see \cite{HLT10,ILT06,LL06,OT07} and references therein for more details.
 \par
 The natural phase space for problem \eqref{7.NS} is $H^{-\bar\gamma}$. This is related with the fact
  that the analogue of the energy estimate holds exactly in this space. Indeed, as known
  $$
  ((u,\Nx v),v)\equiv 0,\ \ u,v\in H^1,
  $$
  so the multiplication of  equation \eqref{7.NS} by $\bar u=(1-\alpha\Dx )^{-\bar\gamma}u$ gives
  \begin{equation}\label{7.en-NS}
  \frac12\frac d{dt}\|(1-\alpha\Dx)^{-\bar\gamma/2}u\|^2_H+
  \|(1-\alpha\Dx)^{-\bar\gamma/2}u\|_{H^{1+\gamma}}^2=(g,(1-\alpha\Dx )^{-\bar\gamma}u)
  \end{equation}
and applying the Gronwall inequality to this relation, we end up with the desired dissipative estimate
 in $H^{-\bar\gamma}$:
 \begin{equation}\label{7.NS-dis}
\|u(t)\|_{H^{-\bar\gamma}}^2+\int_t^{t+1}\|u(s)\|^2_{H^{1+\gamma-\bar\gamma}}\,ds \le Ce^{-\kappa t}\|u_0\|^2_{H^{-\bar\gamma}}+C(1+\|g\|^2_H)
 \end{equation}
for some positive $\kappa$ and $C$. We gave only formal derivation of this estimate, but it can be easily justified, say,
 by the Galerkin method.
\par
The restriction for the uniqueness of a solution and further regularity reads
\begin{equation}\label{7.crit}
2\gamma+\bar\gamma\ge \frac12.
\end{equation}
Indeed, let us indicate how get the uniqueness under this assumption. Let $u_1(t)$ and $u_2(t)$ be two
solutions of \eqref{7.NS} and let $v(t)=u_1(t)-u_2(t)$. Then, writing the equation on $v(t)$ and
 multiplying it by $\bar v:=(1-\alpha\Dx)^{-\bar\gamma}v$, we end up with
 \begin{multline}\label{7.unique}
\frac12\frac d{dt}\|(1-\alpha\Dx)^{-\bar\gamma/2}v\|^2_H+
\|(1-\alpha\Dx)^{-\bar\gamma/2}v\|^2_{H^{1+\gamma}}=\\=
-((v,\Nx)(1-\alpha\Dx)^{-\bar\gamma}u_1,(1-\alpha\Dx)^{-\bar\gamma}v).
 \end{multline}
As an elementary exercise on Sobolev's embeddings and H\"older inequality, one gets
$$
|((v,\Nx)(1-\alpha\Dx)^{-\bar\gamma}u_1,(1-\alpha\Dx)^{-\bar\gamma}v)|\le C\|v\|_{H^{1+\gamma-\bar\gamma}}
\|u_1\|_{H^{1+\gamma-\bar\gamma}}\|v\|_{H^{-\bar\gamma}}
$$
if the criticality assumption \eqref{7.crit} is satisfied (we left the details
 of this exercise to the reader). Inserting this estimate to \eqref{7.unique}, we arrive at
 \begin{equation}
\frac12\frac d{dt}\|(1-\alpha\Dx)^{-\bar\gamma/2}v\|^2_H\le C\|u_1(t)\|_{H^{1+\gamma-\bar\gamma}}^2
\|(1-\alpha\Dx)^{-\bar\gamma/2}v\|_H^2
 \end{equation}
and the Gronwall inequality applied to this relation together with the control of the proper norm
 of $u_1$ obtained in \eqref{7.NS-dis} gives the desired uniqueness.  Note that condition \eqref{7.crit}
  is weaker than our assumption \eqref{7.Img}, so we have verified that equation \eqref{7.NS} generates
   a dissipative semigroup $\bar S(t): H^{-\bar\gamma}\to H^{-\bar\gamma}$.
\par
Moreover, we see that under the condition \eqref{7.Img}, the considered equation has a critical
nonlinearity only in the case $\gamma=0$ and $\bar\gamma=\frac12$. The parabolic smoothing property for this case
is discussed in details in \cite{LS20} (see also \cite{GG18} for other end-point
 case $\gamma=\frac12$, $\bar\gamma=0$), so we will not present this analysis here. The case $\gamma>0$
  is {\it sub-critical} and the further regularity and the existence of absorbing balls in smoother
   spaces are standard corollaries of the {\it linear} parabolic smoothing estimates and
    bootstrapping arguments. So, the actual smoothness of the solution is restricted by the
     smoothness of the
     nonlinearity and external forces only. Under our assumption $g\in H$, we may guarantee
      the existence of
      the absorbing ball in $H^2$ only, but the trick with subtraction of an equilibrium described
       in subsection \ref{ss7.1} allows us to get the desired absorbing ball in $H^s$ with $s<4$. Thus, we
       have proved the following result.

   \begin{corollary}\label{Co7.NS} Let $g\in H$ and the exponents $\gamma$ and $\bar\gamma$
    satisfy \eqref{7.Img}. Then, there are infinitely many values of $N$, such that the
   Navier-Stokes type  equation \eqref{7.NS} possesses  $C^{1+\eb_N}$-smooth IM in the sense of
   Definition \ref{Def5.IMg}. The truncated nonlinearity can be chosen in the form
    of \eqref{5.naive} and \eqref{6.NSnon}
   and  exponents $s_0$ and $s$ satisfying \eqref{6.smoothness}.
   \end{corollary}
\begin{remark}\label{Rem7.NS} The existence of a Lipschitz IM for \eqref{7.NS} in 2D-case with $\gamma=0$,
$\bar\gamma=1$ has been obtained in \cite{HGT15} based on verifying the spectral gap conditions. The
 IM in the
 3D case with $\gamma=0$ and $\bar\gamma=1$ has been constructed in \cite{K18} based on a  novel
 idea to use
 spatial averaging technique for Navier-Stokes type equations (in particular, the special form of the
  cut off function $W(u)$ which is crucial for this approach has been also suggested there).
  The end points $\gamma=0$, $\bar\gamma=\frac12$ and $\gamma=\frac12$, $\bar\gamma=0$ have been
   treated in
  \cite{LS20} and \cite{GG18} respectively. However, in the intermediate case
   $0<\gamma<\frac12$ our
   result seems new. In addition, to the best of our knowledge, the question about $C^{1+\eb}$-smoothness
    of the IMs for the modified Navier-Stokes equations has been never considered before. We also
    emphasize that all the previous partial results for IMS related with equation \eqref{7.NS} as well
     as the new ones are now obtained in a unified way as corollaries of a general theorem.
\end{remark}

\end{document}